\documentclass[11pt]{amsart}

\usepackage{amsmath}
\usepackage{amssymb}
\usepackage{amsthm}
\usepackage[mathscr]{eucal}
\usepackage{mathrsfs}
\usepackage{psfrag}
\usepackage{fullpage}
\usepackage{color}
\usepackage{eucal}
\usepackage[dvips]{graphicx}
\usepackage{amsfonts}%
\usepackage{amsaddr}
\providecommand{\U}[1]{\protect\rule{.1in}{.1in}}

\newtheorem{proposition}{Proposition}[section]
\newtheorem{theorem}[proposition]{Theorem}

\newtheorem{lemma}[proposition]{Lemma}

\newtheorem{definition}[proposition]{Definition}

\newtheorem{example}[proposition]{Example}

\newtheorem{condition}[proposition]{Condition}

\numberwithin{equation}{section}
\numberwithin{proposition}{section}

\begin{document}
\title{Quenched Large Deviations for Multiscale Diffusion Processes in Random Environments}

\author{Konstantinos Spiliopoulos}
\address{Department of Mathematics \& Statistics\\
Boston University\\
Boston, MA 02215}
\email{kspiliop@math.bu.edu}

\date{\today}

\begin{abstract}
We consider multiple time scales systems of stochastic differential equations with small noise in random environments. We prove a quenched large deviations
 principle with explicit characterization of the action functional. The random medium is assumed to be stationary and ergodic. In the course of the proof we also
 prove related quenched ergodic theorems for controlled diffusion processes in random environments that are of independent interest. The proof relies entirely on
 probabilistic arguments, allowing to obtain detailed information on how the rare event occurs. We derive a  control, equivalently a change of measure,
 that leads to the  large deviations lower bound. This information on the change of measure can motivate the design of asymptotically efficient
 Monte Carlo importance sampling schemes for multiscale systems in random environments.
\end{abstract}




\maketitle

\section{Introduction}

\label{S:Intro}

Let $0<\varepsilon,\delta\ll 1$ and consider the process $\left(X^{\epsilon}, Y^{\epsilon}\right)=\left\{\left(X^{\epsilon}_{t}, Y^{\epsilon}_{t}\right), t\in[0,T]\right\}$ taking values in the space $\mathbb{R}^{m}\times\mathbb{R}^{d-m}$ that satisfies the system of stochastic differential equation (SDE's)

\begin{eqnarray}
dX^{\epsilon}_{t}&=&\left[  \frac{\epsilon}{\delta}b\left(Y^{\epsilon}_{t},\gamma\right)+c\left(  X^{\epsilon}_{t}%
,Y^{\epsilon}_{t},\gamma\right)\right]   dt+\sqrt{\epsilon}%
\sigma\left(  X^{\epsilon}_{t},Y^{\epsilon}_{t},\gamma\right)
dW_{t},\nonumber\\
dY^{\epsilon}_{t}&=&\frac{1}{\delta}\left[  \frac{\epsilon}{\delta}f\left(Y^{\epsilon}_{t},\gamma\right)  +g\left(  X^{\epsilon}_{t}%
,Y^{\epsilon}_{t},\gamma\right)\right] dt+\frac{\sqrt{\epsilon}}{\delta}\left[
\tau_{1}\left(  Y^{\epsilon}_{t},\gamma\right)
dW_{t}+\tau_{2}\left(Y^{\epsilon}_{t},\gamma\right)dB_{t}\right], \label{Eq:Main}\\
X^{\epsilon}_{0}&=&x_{0},\hspace{0.2cm}Y^{\epsilon}_{0}=y_{0}\nonumber
\end{eqnarray}
where $\delta=\delta(\epsilon)\downarrow0$ such that $\epsilon/\delta\uparrow\infty$ as $\epsilon\downarrow0$. Here,  $(W_{t}, B_{t})$ is a $2\kappa-$dimensional standard Wiener process. We assume that for each fixed $x\in\mathbb{R}^{m}$,  $b(\cdot,\gamma), c(x,\cdot,\gamma),\sigma(x,\cdot,\gamma),f(\cdot,\gamma)$, $g(x,\cdot,\gamma), \tau_{1}(\cdot,\gamma)$ and $\tau_{2}(\cdot,\gamma)$ are stationary and ergodic random fields. We denote by $\gamma\in\Gamma$ the element of the related probability space. If we want to emphasize the dependence on the initial point and on the random medium, we shall write
$\left(X^{\epsilon,(x_{0},y_{0}),\gamma}, Y^{\epsilon,(x_{0},y_{0}),\gamma}\right)$ for the solution to (\ref{Eq:Main}).

The system (\ref{Eq:Main}) can be interpreted as a small-noise perturbation of dynamical systems with multiple scales. The slow component is $X$ and the fast component is
$Y$.  We study the regime where the homogenization parameter goes faster to zero than the strength of the noise does. The goal of this paper
is to obtain the quenched large deviations principle associated to the component  $X$, that is associated with the slow motion. The case of large deviations for such
 systems in periodic media for all possible interactions between $\epsilon$ and $\delta$, i.e., $\epsilon/\delta\rightarrow 0, c\in(0,\infty)$ or $\infty$, was studied in \cite{Spiliopoulos2013}, see also \cite{Baldi,DupuisSpiliopoulos,FS}. In \cite{Spiliopoulos2013} (see also \cite{DupuisSpiliopoulosWang}), it was assumed that the coefficients
 are periodic with respect to the $y-$variable and based on the derived large deviations principle,  asymptotically efficient importance sampling Monte Carlo methods for estimating rare event probabilities
were obtained. In the current paper, we focus on quenched (i.e. almost sure with respect to the random environment) large deviations for the case $\epsilon/\delta\uparrow\infty$ and
the situation is more complex when compared to the periodic case since the coefficients are now random fields themselves and the fast motion does not take
values in a compact space.

We treat the large deviations problem via the lens of the weak convergence framework, \cite{DupuisEllis}, using entirely probabilistic arguments. This framework transforms the large deviations problem to convergence
of a stochastic control problem. The current work is certainly related to the literature in random homogenization, see
\cite{KomorowskiLandimOlla2012, KosyginaRezakhanlouVaradhan, Kozlov1979,Kozlov1989, LionsSouganidis2006,
Olla1994,OllaSiri2004,Osada1983,Osada1987, PapanicolaouVaradhan1982,Papaanicolaou1994, Rhodes2009a}. Our work is most closely related to \cite{KosyginaRezakhanlouVaradhan,LionsSouganidis2006},
where stochastic homogenization for Hamilton-Jacobi-Bellman (HJB) equations was studied. The authors in \cite{KosyginaRezakhanlouVaradhan, LionsSouganidis2006}
consider
the case $\delta=\epsilon$ with the fast motion being $Y=X/\delta$ and with the coefficients $b=f=0$ in a general Hamiltonian setting. In both papers the authors briefly
discuss large deviations for diffusions (i.e., when the Hamiltonian is quadratic) and the action functional is given as the Legendre-Fenchel transform of
the effective Hamiltonian and the case studied there is $\delta=\epsilon$. Moreover, in \cite{Kushner1, VeretennikovSPA2000} the large deviations principle for systems like (\ref{Eq:Main})
is considered in the case $\epsilon=\delta$ with the coefficients $b=f=0$. In \cite{Kushner1, VeretennikovSPA2000} the coefficients are deterministic (i.e., not random fields as in our case) and stability type conditions
for  the fast process $Y$ are assumed in order to guarantee ergodicity. Lastly, related annealed homogenization results (i.e. on average and not almost sure with respect to the medium)
for uncontrolled multiscale diffusions as in (\ref{Eq:Main}) in the case $\epsilon=1$, $\delta\downarrow 0$ and $Y=X/\delta$ have been recently obtained in \cite{Rhodes2009a}. Under different assumptions on the structure of the coefficients, the opposite case to ours where $\epsilon/\delta\downarrow 0$ has been partially considered in \cite{DupuisSpiliopoulos,FS,Souganidis1999,Spiliopoulos2013}.

In contrast to most of the aforementioned literature, in this paper, we study the case $\epsilon/\delta\uparrow \infty$ and we use entirely probabilistic arguments.
Because $\epsilon/\delta\uparrow \infty$, we also need to consider the additional
 effect of the macroscopic problem (i.e., what is called cell problem in the periodic homogenization literature) due to the highly oscillating term
$\frac{\epsilon}{\delta}\int_{0}^{T}b\left(Y^{\epsilon}_{t},\gamma\right)dt$. We use entirely probabilistic arguments and
because the homogenization parameter goes faster to zero that the
strength of the noise does, we are able to
derive an explicit characterization of the quenched large deviations principle and detailed information on the
change of measure leading to its proof, Theorem \ref{T:MainTheorem3}. Due to the presence of the highly oscillatory term $\frac{\epsilon}{\delta}\int_{0}^{T}b\left(Y^{\epsilon}_{t},\gamma\right)dt$, the change of measure in question depends on the macroscopic problem and we determine this dependence explicitly.
Additionally, in the course of the proof, we obtain quenched (i.e., almost sure with respect to the random environment)  ergodic theorems for uncontrolled and controlled random diffusion processes that may be of independent interest, Theorem \ref{T:MainTheorem1} and Appendix \ref{S:QuenchedErgodicTheorems}. It is of interest to note that for the purposes of proving the Laplace principle, which is equivalent to the large deviations principle, one can constrain the variational problem associated with the stochastic control representation of exponential functionals to a class of $L^{2}$ controls with specific dependence on $\delta,\epsilon$, Lemma \ref{L:RestrictingTheControl}.

Partial motivation for this work comes from chemical physics, molecular dynamics and climate modeling, e.g., \cite{EVMaidaTimofeyev2001,DupuisSpiliopoulosWang2,SchutteWalterHartmannHuisinga2005, Zwanzig}, where one is often interested in simplified models that preserve the large deviation properties of the system in the case where $\delta\ll\epsilon$, i.e., in the case where $\delta$ is orders of magnitude smaller than $\epsilon$.  Other related models where the regime of interest is $\epsilon/\delta\uparrow \infty$
have been considered in \cite{Baldi, DupuisSpiliopoulos, DupuisSpiliopoulosWang, FengFouqueKumar, FS, HorieIshii, Spiliopoulos2013}. When rare events are of interest, then large deviations theory comes into play. As mentioned before,  we are able to derive an explicit characterization of the quenched large deviations principle, Theorem \ref{T:MainTheorem3}.
The explicit form of the derived large deviations action functional and of the control achieving the large deviations bound give useful information
which can be used to design provably efficient importance sampling schemes for estimation of related rare event probabilities. In the case of a periodic fast motion,
the design of large deviations inspired efficient Monte Carlo importance sampling schemes was investigated in
\cite{DupuisSpiliopoulosWang,DupuisSpiliopoulosWang2,Spiliopoulos2013}. The paper \cite{DupuisSpiliopoulosWang} also includes
importance sampling numerical simulations in the case of diffusion moving in a random multiscale environment in dimension one. In the present paper, we
focus on rigorously developing the large deviations theory and the design of asymptotically efficient importance sampling schemes in random environments is
addressed in \cite{Spiliopoulos2014b}.

The rest of the paper is organized as follows. In Section \ref{S:AssumptionMainResult} we set-up notation, state our assumptions and review known results from the literature on random homogenization that will be useful for our purposes. In Section \ref{S:MainResults} we state our main results. Sections \ref{S:ProofLLN}, \ref{S:LDPrelaxedForm} and \ref{S:ProofTheoremExplicitLDP} contain the proofs of the main results of the paper, i.e., quenched homogenization results for pairs of controlled diffusions and occupation measures in random environments  and the large deviations principle with the explicit characterization of the action functional. The Appendix \ref{S:QuenchedErgodicTheorems} contains the proofs of the necessary quenched ergodic theorems for controlled diffusion processes in random environments.

\section{Assumptions, notation and review of useful known results}

\label{S:AssumptionMainResult}

In this section we setup notation and pose the main assumptions of the paper. In this section, and for the convenience of the reader, we also review well known results from the literature on random homogenization that will be useful for our purposes. The content of this section is classical.

We start by describing the properties of the random medium. Let $\left(\Gamma, \mathcal{G},\nu\right)$ be the probability space of the random medium and as in  \cite{JikovKozlovOleinik1994}, a group of measure-preserving transformations $\{\tau_{y},y\in\mathbb{R}^{d}\}$ acting ergodically on $\Gamma$.

\begin{definition}
\label{Def:medium}We assume that the following hold.

\begin{enumerate}
\item {$\tau_{y}$ preserves the measure, namely $\forall y\in\mathbb{R}^{d-m}$
and $\forall A\in\mathcal{G}$ we have $\nu(\tau_{y}A)=\nu(A)$.}

\item {The action of $\{\tau_{y}: y\in\mathbb{R}^{d-m}\}$ is ergodic, that is if
$A=\tau_{y}A$ for every $y\in\mathbb{R}^{d}$ then $\nu(A)=0$ or $1$.}

\item { For every measurable function $f$ on $\left(  \Gamma, \mathcal{G},
\nu\right)  $, the function $(y,\gamma)\mapsto f(\tau_{y}\gamma)$ is
measurable on $\left(  \mathbb{R}^{d-m}\times\Gamma, \mathbb{B}(\mathbb{R}%
^{d-m})\otimes\mathcal{G}\right)  $.}
\end{enumerate}
\end{definition}

For $\tilde{\phi}\in L^{2}(\Gamma)$ (i.e., a square integrable function in $\Gamma$), we define the operator $T_{y}\tilde{\phi}(\gamma)=\tilde{\phi}(\tau_{y}\gamma)$. It is known, e.g. \cite{Olla1994}, that $T_{y}$ forms
a strongly continuous group of unitary maps in $L^{2}(\Gamma)$. Moreover, if the
limit exists, the infinitesimal generator $D_{i}$ of $T_{y}$ in the direction
$i$ is defined by
\begin{equation}
D_{i}\tilde{\phi}=\lim_{h\downarrow0}\frac{T_{he_{i}}\tilde{\phi}-\tilde{\phi}}{h}.
\label{Eq:defofD}%
\end{equation}
and is a closed and densely defined generator.

Next, for $\tilde{\phi}\in L^{2}(\Gamma)$, we define $\phi(y,\gamma)=\tilde{\phi}(\tau_{y}\gamma)$. This definition guarantees that $\phi$
 will be a stationary and ergodic random field on $\mathbb{R}^{d-m}$. Similarly, for a measurable function $\tilde{\phi}:\mathbb{R}^{m}\times\Gamma\mapsto\mathbb{R}^{m}$
we consider the (locally) stationary random field $(x,y) \mapsto \tilde{\phi}(x,\tau_{y}\gamma)=\phi(x,y,\gamma)$.

We follow this procedure to define the random fields $b,c,\sigma,f,g,\tau_{1},\tau_{2}$ that play the role of the coefficients of (\ref{Eq:Main}),
which then guarantees that they are  ergodic and stationary random fields. In particular, we start with $L^{2}(\Gamma)$ functions $\tilde{b}(\gamma),\tilde{c}(x,\gamma),\tilde{\sigma}(x,\gamma),\tilde{f}(\gamma),\tilde{g}(x,\gamma)$, $\tilde{\tau}_{1}(\gamma),\tilde{\tau}_{2}(\gamma)$ and we define the coefficients of (\ref{Eq:Main}) via the relations
$b(y,\gamma)=\tilde{b}(\tau_{y}\gamma),c(x,y,\gamma)=\tilde{c}(x,\tau_{y}\gamma),\sigma(x,y,\gamma)=\tilde{\sigma}(x,\tau_{y}\gamma),f(y,\gamma)=\tilde{f}(\tau_{y}\gamma),
g(x,y,\gamma)=\tilde{g}(x,\tau_{y}\gamma),\tau_{1}(y,\gamma)=\tilde{\tau}_{1}(\tau_{y}\gamma)$ and $\tau_{2}(y,\gamma)=\tilde{\tau}_{2}(\tau_{y}\gamma)$.

The main assumption for the coefficients of (\ref{Eq:Main}) is as follows.

\begin{condition}
\label{A:Assumption1}

\begin{enumerate}
\item The functions $b(y,\gamma),c(x,y,\gamma),\sigma(x,y,\gamma), f(y,\gamma), g(x,y,\gamma), \tau_{1}(y,\gamma)$ and $\tau_{2}(y,\gamma)$ are
$C^{1}(\mathbb{R}^{d-m})$ in $y$ and $C^{1}(\mathbb{R}^{m})$ in $x$ with all
partial derivatives continuous and globally bounded in $x$ and $y $.

\item For every fixed $\gamma\in\Gamma$, the diffusion matrices $\sigma\sigma^{T}$ and $\tau_{1}\tau_{1}^{T}+\tau_{2}\tau_{2}^{T}$ are uniformly nondegenerate.
\end{enumerate}
\end{condition}

It is known that under Condition \ref{A:Assumption1}, there exists a filtered probability space $(\Omega,\mathcal{F},\mathfrak{F}_{t},\mathbb{P})$ such that for every given initial point $(x_{0},y_{0})\in\mathbb{R}^{m}\times\mathbb{R}^{d-m}$, for every $\gamma\in\Gamma$ and for every $\epsilon,\delta>0$ there exists a strong Markov process $\left(X^{\epsilon}_{t},Y^{\epsilon}_{t}, t\geq 0\right)$ satisfying (\ref{Eq:Main}). However, if we define a probability measure $\mathcal{P}=\nu\otimes\mathbb{P}$ on the product space
$\Gamma\times\Omega$, then when considered
on the probability space $(\Gamma\times\Omega,\mathcal{G}\otimes
\mathcal{F},\mathcal{P})$, $\{\left(X_{t}^{\epsilon}, Y^{\epsilon}_{t}\right),t\geq0\}$ is not a Markov process.

From the previous discussion it is easy to see that the periodic case is a special case of the previous setup. Indeed, we can consider  the periodic case with period $1$, $\Gamma$ to be the unit
torus and $\nu$ to be Lebesgue measure on $\Gamma$. For every
$\gamma\in\Gamma$, the shift operators $\tau_{y}\gamma
=(y+\gamma)\mod 1$ and we have $\phi(y,\gamma)=\tilde{\phi}(y+\gamma)$ for a periodic
function $\tilde{\phi}$ with period $1$.

For every $\gamma\in\Gamma$, we define next the operator
\[
\mathcal{L}^{\gamma}=f(y,\gamma)\nabla_{y}\cdot+\text{\emph{tr}}\left[\left(
\tau_{1}(y,\gamma)\tau^{T}_{1}(y,\gamma)+\tau_{2}(y,\gamma)\tau^{T}_{2}(y,\gamma)\right)\nabla_{y}\nabla_{y}\cdot\right] \label{Def:OperatorFastProcess}
\]
and we let $Y_{t}^{\gamma}$ to be the corresponding Markov process. It follows from \cite{PapanicolaouVaradhan1982,Osada1983,Olla1994}, that we can associate the
canonical process on $\Gamma$ defined by the environment $\gamma$, which is a Markov process on $\Gamma$ with continuous
transition probability densities with respect to $d$-dimensional Lebesgue
measure, e.g., \cite{Olla1994}. In particular, we let
\begin{align}
\gamma_{t}  &  =\tau_{Y_{t}^{\gamma}}\gamma\label{Eq:EnvironmentProcess}\\
\gamma_{0}  &  =\tau_{y_{0}}\gamma\nonumber
\end{align}

\begin{definition}
\label{Def:OperatorFastProcess2} We denote the infinitesimal generator of the
Markov process $\gamma_{t}$ by
\[
\tilde{L}=\tilde{f}(\gamma)D\cdot+\text{\emph{tr}}\left[ \left( \tilde{\tau
}_{1}(\gamma)\tilde{\tau}_{1}^{T}(\gamma)+\tilde{\tau
}_{2}(\gamma)\tilde{\tau}_{2}^{T}(\gamma)\right)D^{2}\cdot\right]  ,
\]
where $D$ was defined in (\ref{Eq:defofD}).
\end{definition}

Following \cite{Osada1983}, we assume the following condition on the structure of the operator defined in Definition
\ref{Def:OperatorFastProcess2}. This condition allows to have a closed form for the unique ergodic invariant measure for the environment process
$\{\gamma_{t}\}_{t\geq0}$, Proposition \ref{P:NewMeasureRandomCase}.

\begin{condition}
\label{A:Assumption2} We can write the operator $\tilde{L}$ in the following generalized divergence form
\[
\tilde{L}=\frac{1}{\tilde{m}(\gamma)}\left[  \sum_{i,j}D_{i}\left(  \tilde
{a}\tilde{a}_{i,j}^{T}(\gamma)D_{j}\cdot\right)  +\sum_{j}\tilde{\beta}%
_{j}(\gamma)D_{j}\cdot\right]
\]
where $\tilde{\beta}_{j}=\tilde{m}\tilde{f}_{j}-\sum_{i}D_{i}\left(
\left(\tilde{\tau}_{1}\tilde{\tau}_{1}^{T}+\tilde{\tau}_{2}\tilde{\tau}_{2}^{T}\right)_{i,j}\tilde{m}\right)  $ and $\tilde{a}%
\tilde{a}_{i,j}^{T}=
\left(\tilde{\tau}_{1}\tilde{\tau}_{1}^{T}+\tilde{\tau}_{2}\tilde{\tau}_{2}^{T}\right)_{i,j}\tilde{m}$. We
assume that $\tilde{m}(\gamma)$ is bounded from below and from above with
probability $1$, that there exist smooth $\tilde{d}_{i,j}(\gamma)$ such that
$\tilde{\beta}_{j}=\sum_{j}D_{j}\tilde{d}_{i,j}$ with $|\tilde{d}_{i,j}|\leq
M$ for some $M<\infty$ almost surely and
\[
\text{div }\tilde{\beta}=0\text{ in distribution},\quad\text{i.e.,}%
\int_{\Gamma}\sum_{j=1}^{d}\tilde{\beta}_{j}(\gamma)D_{j}\tilde{\phi}%
(\gamma)\nu(d\gamma)=0,\quad\forall\tilde{\phi}\in\mathcal{H}^{1},
\]
where the Sobolev space $\mathcal{H}^{1}=\mathcal{H}^{1}(\nu)$ is the Hilbert space  equipped with the
inner product
\[
(\tilde{f},\tilde{g})_{1}=\sum_{i=1}^{d}(D_{i}\tilde
{f},D_{i}\tilde{g}).
\]

\end{condition}
\begin{example}
A trivial example that satisfies Condition \ref{A:Assumption2} is the gradient case. Let $\tilde{f}(\gamma)=-D
\tilde{Q}(\gamma)$ and $\tilde{\tau}_{1}(\gamma)=\sqrt{2D}=\text{constant}$ and $\tilde{\tau}_{2}(\gamma)=0$. Then, we have  that
$\tilde{m}(\gamma)=\exp[-\tilde{Q}(\gamma)/D]$ and $\tilde{\beta}_{j}=0$
for all $1\leq j\leq d$. Moreover, if $\tilde{m}=1$ and
$\tilde{d}_{i,j}$ are constants then the operator is of divergence form.
\end{example}

Next, we recall some classical results from random homogenization.
\begin{proposition}
[\cite{Osada1983} and Theorem 2.1 in \cite{Olla1994}]%
\label{P:NewMeasureRandomCase} Assume Conditions \ref{A:Assumption1} and
\ref{A:Assumption2}. Define a measure on $(\Gamma,\mathcal{G})$ by%
\[
\pi(d\gamma)\doteq\frac{\tilde{m}(\gamma)}{\mathrm{E}^{\nu}\tilde{m}(\cdot
)}\nu(d\gamma).
\]
Then $\pi$ is the unique ergodic invariant measure for the environment process
$\{\gamma_{t}\}_{t\geq0}$.
\end{proposition}

We will denote by $\mathrm{E}^{\nu}$ and by $\mathrm{E}^{\pi}$ the expectation
operator with respect to the measures $\nu$ and $\pi$ respectively. We remark here that since $\tilde{m}$ is bounded from above and from below,
$\mathcal{H}^{1}(\nu)$ and $\mathcal{H}^{1}(\pi)$ are equivalent.  We also need to introduce
 the macroscopic problem, known as cell problem in the periodic homogenization literature or corrector in the homogenization literature in
general. This is needed in order to address the situation $\tilde{b}\neq0$. For every $\rho>0$, we consider the solution to the auxiliary problem on $\Gamma$.
\begin{equation}
\rho\tilde{\chi}_{\rho}-\tilde{L}\tilde{\chi}_{\rho}=\tilde{b}.
\label{Eq:RandomCellProblem}%
\end{equation}

Let us review some well known facts related to the solution to this auxiliary problem, e.g.,
see \cite{Olla1994, KomorowskiLandimOlla2012}. By Lax-Milgram lemma,
equation (\ref{Eq:RandomCellProblem}) has a unique weak solution in the abstract Sobolev space
$\mathcal{H}^{1}$. Moreover, letting $\mathcal{R}_{\rho}\tilde
{h}(\gamma)=\int_{0}^{\infty}e^{-\rho t}\mathrm{E}_{\gamma}\tilde{h}%
(\gamma_{t})dt$, for every $\tilde{h}\in L^{2}(\Gamma)$, we have
\[
\tilde{\chi}_{\rho}(\cdot)=\mathcal{R}_{\rho}\tilde{b}(\cdot),
\]

As in \cite{Osada1983,PapanicolaouVaradhan1982}, there is a constant $K$ that is independent of $\rho$ such that
\[
\rho\mathrm{E}^{\pi}\left[  \tilde{\chi}_{\rho}(\cdot)\right]  ^{2}%
+\mathrm{E}^{\pi}\left[  D\tilde{\chi}_{\rho}(\cdot)\right]  ^{2}\leq K
\]

By Proposition 2.6 in
\cite{Olla1994} we then get that $\tilde{\chi}_{\rho}$ has an $\mathcal{H}^{1}$
strong limit, i.e., there exists a $\tilde{\chi}_{0}\in\mathcal{H}^{1}(\pi)$
such that
\[
\lim_{\rho\downarrow0}\left\Vert \tilde{\chi}_{\rho}(\cdot)-\tilde{\chi}%
_{0}(\cdot)\right\Vert _{1}=0
\]
and that
\[
\lim_{\rho\downarrow0}\rho\mathrm{E}^{\pi}\left[  \tilde{\chi}_{\rho}%
(\cdot)\right]  ^{2}=0.
\]

This implies that $D\tilde{\chi}_{\rho}\in L^{2}(\pi)$ and that it has a $L^{2}(\pi)$
strong limit, i.e., there exists a $\tilde{\xi}\in L^{2}(\pi)$ such that
\[
\lim_{\rho\downarrow0}\left\Vert D\tilde{\chi}_{\rho}-\tilde{\xi}\right\Vert
_{L^{2}}^{2}=0
\]

In addition, since $\tilde{b}$ is bounded under Condition \ref{A:Assumption1},
$\tilde{\chi}_{\rho}$ is also bounded. This follows because the resolvent
operator $\mathcal{R}_{\rho}$ corresponding to the operator $\rho
I-\mathcal{L}$ is associated to a $L^{\infty}(\Gamma)$ contraction semigroup,
see Section 2.2 of \cite{Olla1994}.

Moreover, as in Proposition 3.2. of \cite{Osada1983}, we have that for almost all $\gamma\in\Gamma$
\[
\delta \chi_{0}(y/\delta,\gamma)\rightarrow 0, \text{ as }\delta\downarrow 0, \text{ a.s. } y\in\mathcal{Y}.
\]

\section{Main results}\label{S:MainResults}

In this section we present the statement of the main results of the paper. In preparation for stating the large deviations theorem, we first recall the concept of a Laplace principle.

\begin{definition}
\label{Def:LaplacePrinciple} Let $\{X^{\epsilon},\epsilon>0\}$ be a family of
random variables taking values in a Polish space $\mathcal{S}$ and let $I$ be a rate function
on $\mathcal{S}$. We say that $\{X^{\epsilon},\epsilon>0\}$ satisfies the
Laplace principle with rate function $I$ if for every bounded and continuous
function $h:\mathcal{S}\rightarrow\mathbb{R}$
\begin{equation}
\lim_{\epsilon\downarrow0}-\epsilon\ln\mathbb{E}\left[  \exp\left\{
-\frac{h(X^{\epsilon})}{\epsilon}\right\}  \right]  =\inf_{x\in\mathcal{S}%
}\left[  I(x)+h(x)\right]. \label{Eq:LaplacePrinciple}%
\end{equation}

\end{definition}

If the level sets of the rate function (equivalently action functional) are compact, then the
Laplace principle is equivalent to the corresponding large deviations
principle with the same rate function (Theorems 2.2.1 and 2.2.3
in \cite{DupuisEllis}).

In order to establish the quenched Laplace principle, we make use of the representation theorem for functionals of the form $\mathbb{E}\left[e^{-\frac{1}{\epsilon}h\left(X^{\epsilon,\gamma}\right)}\right]$ in terms of a stochastic control problem. Such representations were first derived in \cite{BoueDupuis}.

Let $\mathcal{A}$ be the set of all $\mathfrak{F}_{s}-$progressively
measurable $n$-dimensional processes $u\doteq\{u(s),0\leq s\leq T\}$
satisfying
\begin{equation*}
\mathbb{E}\int_{0}^{T}\left\Vert u(s)\right\Vert ^{2}ds<\infty,
\label{A:AdmissibleControls}%
\end{equation*}
In the present case, let $Z(\cdot)=(W(\cdot),B(\cdot))$ and $n=2k$. Then, for the given $\gamma\in\Gamma$ we have the representation

\begin{equation}
-\epsilon\ln\mathbb{E}_{x_{0}, y_{0}}
\left[  \exp\left\{  -\frac{h(X^{\epsilon}%
)}{\epsilon}\right\}  \right]  =\inf_{u\in\mathcal{A}}\mathbb{E}_{x_{0},y_{0}%
}\left[  \frac{1}{2}\int_{0}^{T}\left[\left\Vert u_{1}(s)\right\Vert ^{2}+\left\Vert u_{2}(s)\right\Vert ^{2}\right]ds+h(\bar
{X}^{\epsilon})\right] \label{Eq:VariationalRepresentation}
\end{equation}
where the pair $(\bar{X}^{\epsilon},\bar{Y}^{\epsilon})$ is the unique strong solution to

\begin{eqnarray}
d\bar{X}^{\epsilon}_{t}&=&\left[  \frac{\epsilon}{\delta}b\left(\bar{Y}^{\epsilon}_{t},\gamma\right)+c\left(  \bar{X}^{\epsilon}_{t}%
,\bar{Y}^{\epsilon}_{t},\gamma\right)+\sigma\left(  \bar{X}_{t}^{\epsilon},\bar{Y}_{t}^{\epsilon},\gamma\right)  u_{1}(t)\right]   dt+\sqrt{\epsilon}%
\sigma\left(  \bar{X}^{\epsilon}_{t},\bar{Y}^{\epsilon}_{t},\gamma\right)
dW_{t}, \nonumber\\
d\bar{Y}^{\epsilon}_{t}&=&\frac{1}{\delta}\left[  \frac{\epsilon}{\delta}f\left(\bar{Y}^{\epsilon}_{t},\gamma\right)  +g\left(  \bar{X}^{\epsilon}_{t}
,\bar{Y}^{\epsilon}_{t},\gamma\right)+\tau_{1}\left(\bar{Y}^{\epsilon}_{t},\gamma\right)u_{1}(t)+
\tau_{2}\left(\bar{Y}^{\epsilon}_{t},\gamma\right)u_{2}(t)\right]   dt\nonumber\\
& &\hspace{5cm}+\frac{\sqrt{\epsilon}}{\delta}\left[
\tau_{1}\left(\bar{Y}^{\epsilon}_{t},\gamma\right)
dW_{t}+\tau_{2}\left(\bar{Y}^{\epsilon}_{t},\gamma\right)dB_{t}\right],\label{Eq:Main2}\\
\bar{X}^{\epsilon}_{0}&=&x_{0},\hspace{0.2cm}\bar{Y}^{\epsilon}_{0}=y_{0}\nonumber
\end{eqnarray}

This representation implies that in order to derive the Laplace principle for
$\{X^{\epsilon}\}$, it is enough to study the limit of the right
hand side of the variational representation
(\ref{Eq:VariationalRepresentation}). The first step in doing so is
to consider the weak limit of the slow motion $\bar{X}^{\epsilon}$
of the controlled couple (\ref{Eq:Main2}).

Fix $\gamma\in\Gamma$ and let us define for notational convenience $\mathcal{Z}=\mathbb{R}^{\kappa}$ and $\mathcal{Y}=\mathbb{R}^{d-m}$. Due to the involved
controls, it is convenient to introduce the following occupation
measure. Let $\Delta=\Delta(\epsilon )\downarrow0$ as
$\epsilon\downarrow0$ that will be chosen later on and is used to exploit a time-scale separation. Let $A_{1},A_{2},B,\Theta$ be Borel sets of
$\mathcal{Z},\mathcal{Z},\Gamma,[0,T]$ respectively. Let
$u^{\epsilon}_{i}\in A_{i}, i=1,2$ and let $(\bar{X}^{\epsilon
},\bar{Y}^{\epsilon})$ solve (\ref{Eq:Main2}) with
$u^{\epsilon}_{i}$ in place of $u_{i}$. We associate with
$(\bar{X}^{\epsilon},\bar{Y}^{\epsilon})$ and
$u^{\epsilon}_{i}$ a family of occupation measures
$\mathrm{P}^{\epsilon,\Delta,\gamma}$ defined by
\begin{equation*}
\mathrm{P}^{\epsilon,\Delta,\gamma}(A_{1}\times A_{2}\times B\times\Theta)=\int_{\Theta}\left[
\frac{1}{\Delta}\int_{t}^{t+\Delta}1_{A_{1}}(u_{1}^{\epsilon}(s))1_{A_{2}}(u_{2}^{\epsilon}(s))1_{B}\left(
\tau_{\bar{Y}_{s}^{\epsilon}}\gamma\right)  ds\right]  dt,
\label{Def:OccupationMeasures2}%
\end{equation*}
assuming that $u_{i}^{\epsilon}(t)=0$ for $i=1,2$ if $t>T$. Next, we introduce the notion of a viable pair, see also \cite{DupuisSpiliopoulos}. Such a notion will allow us to  characterize the limiting behavior of the pair $\left(\bar{X}^{\epsilon,\gamma}, \mathrm{P}^{\epsilon,\Delta,\gamma}\right)$.

\begin{definition}
\label{D:ViablePair} Define the function in $L^{2}(\Gamma)$
\[
\tilde{\lambda}(x,\gamma,z_{1},z_{2})   = \tilde{c}(x,\gamma)+\tilde{\xi}(\gamma) \tilde{g}(x,\gamma) +\tilde{\sigma}(x,\gamma)z_{1}+\tilde{\xi}(\gamma)\left(\tilde{\tau}_{1}(\gamma)z_{1}+\tilde{\tau}_{2}(\gamma)z_{2}\right)
\]
where $\tilde{\xi}$ is the $L^{2}$ limit of $D\tilde{\chi}_{\rho}$ as $\rho\downarrow 0$ that is defined in Section \ref{S:AssumptionMainResult}. Consider the operator
$\tilde{L}$ defined in Definition \ref{Def:OperatorFastProcess2}. We say that
a pair $(\psi,\mathrm{P})\in\mathcal{C}\left(  [0,T];\mathbb{R}^{m}\right)
\times\mathcal{P}\left(  \mathcal{Z}\times\mathcal{Z}\times\Gamma\times\lbrack0,T]\right)  $ is
viable with respect to $(\tilde{\lambda},\tilde{L})$ and we write
$(\psi,\mathrm{P})\in\mathcal{V}$, if the following hold.

\begin{itemize}
\item { The function $\psi$ is absolutely continuous and $\mathrm{P}$ is
square integrable in the sense that $\int_{\mathcal{Z}\times\mathcal{Z}\times\Gamma\times
[0,T]}|z|^{2}\mathrm{P}(dz_{1}dz_{2} d\gamma dt)<\infty$.}

\item {For all $t\in[0,T]$, $\mathrm{P}\left(  \mathcal{Z}\times\mathcal{Z}\times\Gamma
\times[0,t]\right)  =t$. Thus,  $\mathrm{P}$ can be decomposed as
$\mathrm{P}(dz_{1}dz_{2} d\gamma dt)=\mathrm{P}_{t}(dz_{1}dz_{2} d\gamma)dt$ such that
$\mathrm{P}_{t}(\mathcal{Z}\times\mathcal{Z}\times\Gamma)=1$. }

\item {For all $t\in\lbrack0,T]$, $\left(\psi,\mathrm{P}\right)$ satisfy the ODE
\begin{equation}
\psi_{t}=x_{0}+\int_{0}^{t}\left[  \int_{\mathcal{Z}\times\mathcal{Z}%
\times\Gamma}\tilde{\lambda}(\psi_{s},\gamma,z_{1},z_{2})\mathrm{P}_{s}%
(dz_{1}dz_{2}d\gamma)\right]  ds. \label{Eq:ViablePair1}%
\end{equation}
and for a given $\mathrm{P}$, there is a unique well defined $\psi$ satisfying (\ref{Eq:ViablePair1}).
}

\item {For a.e. $t\in\lbrack0,T]$,
\begin{equation}
\int_{\mathcal{Z}\times\mathcal{Z}\times\Gamma}\tilde{L}\tilde{f}(\gamma)\mathrm{P}%
_{t}(dz_{1}dz_{2}d\gamma)=0 \label{Eq:ViablePair2}%
\end{equation}
for all $\tilde{f}\in\mathcal{D}(\tilde{L})$.}
\end{itemize}
\end{definition}
For notational convenience later on, let us also define
\[
\tilde{\lambda}_{\rho}(x,\gamma,z_{1},z_{2})   = \tilde{c}(x,\gamma)+D\tilde{\chi}_{\rho}(\gamma) \tilde{g}(x,\gamma) +\tilde{\sigma}(x,\gamma)z_{1}+D\tilde{\chi}_{\rho}(\gamma)\left(\tilde{\tau}_{1}(\gamma)z_{1}+\tilde{\tau}_{2}(\gamma)z_{2}\right)
\]

Now, that we have defined the notion of a viable pair we are ready to present the law of large numbers results for controlled pairs $\left(\bar{X}^{\epsilon,\gamma},\mathrm{P}^{\epsilon,\Delta,\gamma}\right)$.

\begin{theorem}
\label{T:MainTheorem1} Assume Conditions \ref{A:Assumption1} and \ref{A:Assumption2}. Fix the initial point
 $(x_{0},y_{0})\in\mathbb{R}^{m}\times \mathcal{Y}$ and consider a family
$\{u^{\epsilon}=(u^{\epsilon}_{1},u^{\epsilon}_{2}),\epsilon>0\}$ of controls (that may depend on $\gamma$) in $\mathcal{A}$ satisfying a.s. with respect to $\gamma\in\Gamma$, the bound \ref{Eq:UniformlySquareIntegrableControlsAdditionalb} and
\begin{equation}
\sup_{\epsilon>0}\mathbb{E}\int_{0}^{T}\left[\left\Vert u_{1}^{\epsilon}(s)\right\Vert
^{2}+\left\Vert u_{2}^{\epsilon}(s)\right\Vert
^{2}\right]ds<\infty\label{Eq:Ubound}%
\end{equation}

Then the family $\{(\bar{X}^{\epsilon,\gamma
},\mathrm{P}^{\epsilon,\Delta,\gamma}),\epsilon>0\}$ is tight almost surely with respect to $\gamma\in\Gamma$. Given  any subsequence of $\{(\bar{X}^{\epsilon}%
,\mathrm{P}^{\epsilon,\Delta}),\epsilon>0\}$, there exists a
subsubsequence that converges in distribution with limit
$(\bar{X},\mathrm{P})$ almost surely with respect to $\gamma\in\Gamma$. With probability $1$, the limit
point
$(\bar{X},\mathrm{P})\in\mathcal{V}$,
according to Definition \ref{D:ViablePair}.
\end{theorem}

Next, we are ready to state the quenched Laplace principle for $\left\{X^{\epsilon},\epsilon>0\right\}$.

\begin{theorem}
\label{T:MainTheorem2} Let $\{\left(X^{\epsilon},Y^{\epsilon}\right),\epsilon>0\}$ be, for fixed $\gamma\in\Gamma$,  the
unique strong solution to (\ref{Eq:Main}) and assume that $\epsilon/\delta\uparrow \infty$. We assume that Conditions
\ref{A:Assumption1} and  \ref{A:Assumption2} hold. Define
\begin{equation}
S(\phi)=\inf_{(\phi,\mathrm{P})\in\mathcal{V}}\left[  \frac{1}{2}\int_{\mathcal{Z}\times\mathcal{Z}\times\mathcal{Y}\times\lbrack
0,T]}\left[\left\Vert z_{1}\right\Vert ^{2}+\left\Vert z_{2}\right\Vert ^{2}\right]\mathrm{P}(dz_{1}dz_{2}dydt)\right]  ,
\label{Eq:GeneralRateFunction}%
\end{equation}
with the convention that the infimum over the empty set is $\infty$. Then, we have
\begin{enumerate}
 \item The level sets of $S$ are compact. In particular, for each $s<\infty$, the set
\begin{equation*}
\Phi_{s}=\{\phi\in\mathcal{C}([0,T];\mathbb{R}^{m}):S(\phi)\leq s\}
\label{Def:LevelSets}%
\end{equation*}
is a compact subset of $\mathcal{C}([0,T];\mathbb{R}^{m})$.
\item For
every bounded and continuous function $h$ mapping $\mathcal{C}%
([0,T];\mathbb{R}^{m})$ into $\mathbb{R}$
\begin{equation*}
\lim_{\epsilon\downarrow0}-\epsilon\ln\mathbb{E}_{x_{0},y_{0}}\left[  \exp\left\{
-\frac{h(X^{\epsilon,\gamma})}{\epsilon}\right\}  \right]  = \inf_{\phi\in
\mathcal{C}([0,T];\mathbb{R}^{m}), \phi_{0}=x_{0}}\left[  S(\phi)+h(\phi)\right]  .
\end{equation*}
almost surely with respect to $\gamma\in\Gamma$.
\end{enumerate}

In other words, under the imposed assumptions,
$\{X^{\epsilon,\gamma},\epsilon>0\}$ satisfies the quenched large deviations principle with action functional $S$.
\end{theorem}

Actually, it turns out that in this case we can compute the quenched action functional in closed form.

\begin{theorem}
\label{T:MainTheorem3} Let $\{\left(X^{\epsilon,\gamma},Y^{\epsilon,\gamma}\right),\epsilon>0\}$ be, for fixed $\gamma\in\Gamma$, the unique strong
solution to (\ref{Eq:Main}). Under Conditions
\ref{A:Assumption1} and \ref{A:Assumption2}, $\{X^{\epsilon,\gamma},\epsilon>0\}$
satisfies, almost surely with respect to $\gamma\in\Gamma$, the large deviations principle with rate function
\begin{equation*}
S(\phi)=%
\begin{cases}
\frac{1}{2}\int_{0}^{T}(\dot{\phi}(s)-r(\phi(s)))^{T}q^{-1}(\phi(s)%
)(\dot{\phi}(s)-r(\phi(s)))ds & \text{if }\phi\in\mathcal{AC}%
([0,T];\mathbb{R}^{m}) \text{  and } \phi(0)=x_{0}\\
+\infty & \text{otherwise.}%
\end{cases}
\label{Eq:ActionFunctional1}%
\end{equation*}
where
\[
r(x)=\lim_{\rho\downarrow0}\mathrm{E}^{\pi}\left[  \tilde{c}(x,\cdot) +D\tilde{\chi}_{\rho
}(\cdot) \tilde{g}(x,\cdot)\right]  =\mathrm{E}^{\pi}[\tilde{c}(x,\cdot)+\tilde{\xi}%
(\cdot)\tilde{g}(x,\cdot)]
\]%
\begin{align}
q(x)&=\lim_{\rho\downarrow0}\mathrm{E}^{\pi}\left[  (\tilde{\sigma}(x,\cdot)+D\tilde{\chi}_{\rho}%
(\cdot)\tilde{\tau}_{1}(\cdot))(\tilde{\sigma}(x,\cdot)+D\tilde{\chi}_{\rho}%
(\cdot)\tilde{\tau}_{1}(\cdot))^{T}+\left(D\tilde{\chi}_{\rho}%
(\cdot)\tilde{\tau}_{2}(\cdot)\right)\left(D\tilde{\chi}_{\rho}%
(\cdot)\tilde{\tau}_{2}(\cdot)\right)^{T}\right]\nonumber\\
&=\mathrm{E}^{\pi}\left[  (\tilde{\sigma}(x,\cdot)+\tilde{\xi}(\cdot)\tilde{\tau}_{1}(\cdot))(\tilde{\sigma}(x,\cdot)+\tilde{\xi}
(\cdot)\tilde{\tau}_{1}(\cdot))^{T}+\left(\tilde{\xi}
(\cdot)\tilde{\tau}_{2}(\cdot)\right)\left(\tilde{\xi}
(\cdot)\tilde{\tau}_{2}(\cdot)\right)^{T}\right]\nonumber
\end{align}
\end{theorem}

Notice that the coefficients $r(x)$ and $q(x)$ that enter into the
action functional are those obtained if we had first
taken to (\ref{Eq:Main}) $\delta\downarrow 0$ with $\epsilon$ fixed
and then consider the large deviations for the homogenized system.
This is in accordance to intuition since in the case $\epsilon/\delta\uparrow\infty$, $\delta$ goes to zero faster than $\epsilon$. This implies that homogenization should occur first as it indeed does and then large deviations start playing a role.

\section{Proof of Theorem \ref{T:MainTheorem1}}\label{S:ProofLLN}

In this section we prove Theorem \ref{T:MainTheorem1}. Tightness is established in Subsection \ref{S:Tightness}, whereas the identification of the limit point is done in Subsection \ref{S:WeakConvergence}.

\subsection{Tightness of the controlled pair $\left\{  (\bar{X}^{\epsilon,\gamma
},\mathrm{P}^{\epsilon,\Delta,\gamma}), \epsilon,\Delta>0 \right\}  $.}

\label{S:Tightness}

In this section we prove that the family $\{(\bar{X}^{\epsilon,\gamma
},\mathrm{P}^{\epsilon,\Delta,\gamma}),\epsilon>0\}$, is almost surely tight with respect to $\gamma\in\Gamma$ where $\Delta
=\Delta(\epsilon)\downarrow0$. The following proposition takes care of
tightness and uniform integrability of $\{\mathrm{P}^{\epsilon,\Delta,\gamma
},\epsilon>0\}$.

\begin{lemma}
\label{L:TightnessOccupationalMeasures} Assume Conditions \ref{A:Assumption1} and \ref{A:Assumption2}.
Let $\{u^{\epsilon,\gamma},\epsilon>0,\gamma\in\Gamma\}$ be a family of
controls in $\mathcal{A}$ such that Conditions \ref{Eq:UniformlySquareIntegrableControlsAdditional} and \ref{Eq:UniformlySquareIntegrableControlsAdditionalb} of Lemma \ref{L:ErgodicTheorem4} hold.
The following hold

\begin{enumerate}
\item {For every $\eta>0$, there is a set $N_{\eta}$ (the same $N_{\eta}$ identified in Lemma \ref{L:ErgodicTheorem4}) with $\pi(N_{\eta})\geq 1-\eta$ such that for every $\gamma\in N_{\eta}$ and for every bounded sequence $\Delta\in \mathcal{H}^{N_{\eta}}_{1}$ (i.e. a sequence that satisfies Condition \ref{A:h_functionUniform}), the family $\{\mathrm{P}^{\epsilon,\Delta,\gamma},\epsilon>0\}$ is
tight as $\epsilon\downarrow 0$.}

\item {The family $\{\mathrm{P}^{\epsilon,\Delta,\gamma},\epsilon>0\}$ is
uniformly integrable, in the sense that
\[
\lim_{M\rightarrow\infty}\sup_{\epsilon>0,\gamma\in\Gamma}\mathbb{E}\int_{\{(z_{1},z_{1})\in\mathcal{Z}^{2}:\left[\left\Vert z_{1}\right\Vert+\left\Vert z_{2}\right\Vert\right]\geq M\}\times\Gamma\times\lbrack
0,T]}\left[\left\Vert z_{1}\right\Vert+\left\Vert z_{2}\right\Vert\right]\mathrm{P}^{\epsilon,\Delta,\gamma}(dz_{1}dz_{1}d\tilde{\gamma}dt)=0
\]
}
\end{enumerate}
\end{lemma}

\begin{proof}
(i). Let us first prove the first part of the Lemma. It is clear that we can write
\[
\mathrm{P}^{\epsilon,\Delta,\gamma}(A_{1}\times A_{2}\times B\times\Theta)=\int_{\Theta}\mathrm{P}^{\epsilon,\Delta,\gamma}_{t}(A_{1}\times A_{2}\times B)dt
\]
where
\[
\mathrm{P}^{\epsilon,\Delta,\gamma}_{t}(A_{1}\times A_{2}\times B)=\left[  \frac{1}{\Delta}\int_{t}^{t+\Delta}1_{A_{1}}(u_{1}^{\epsilon,\gamma
}(s))1_{A_{2}}(u_{2}^{\epsilon,\gamma}(s))1_{B}\left(  \tau_{\bar{Y}_{s}^{\epsilon}}\gamma\right)
ds\right]  dt,
\]

Let us denote by $\mathrm{P}^{\epsilon,\Delta,\gamma}_{1,t}(A_{1}\times A_{2})$ and by $\mathrm{P}^{\epsilon,\Delta,\gamma}_{2,t}( B)$ the first and second marginals of $\mathrm{P}^{\epsilon,\Delta,\gamma}_{t}(A_{1}\times A_{2}\times B)$ respectively. Namely,
\[
\mathrm{P}^{\epsilon,\Delta,\gamma}_{1,t}(A_{1}\times A_{2})=\mathrm{P}^{\epsilon,\Delta,\gamma}_{t}(A_{1}\times A_{2}\times \Gamma),\text{ and }\mathrm{P}^{\epsilon,\Delta,\gamma}_{2,t}(B)=\mathrm{P}^{\epsilon,\Delta,\gamma}_{t}(\mathcal{Z}\times \mathcal{Z}\times B)
\]

It is clear that tightness of $\{\mathrm{P}^{\epsilon,\Delta,\gamma},\epsilon>0\}$ is a consequence of tightness of
 $\{\mathrm{P}^{\epsilon,\Delta,\gamma}_{1,t},\epsilon>0\}$ and of $\{\mathrm{P}^{\epsilon,\Delta,\gamma}_{2,t},\epsilon>0\}$.

Let us first consider tightness of $\{\mathrm{P}^{\epsilon,\Delta,\gamma}_{1,t},\epsilon>0\}$. For this purpose, we claim that the function
\[
g(r)=\int_{\mathcal{Z}\times\mathcal{Z}\times\lbrack0,T]}\left[\left\Vert
z_{1}\right\Vert ^{2}+\left\Vert
z_{2}\right\Vert ^{2}\right]r(dz_{1}dz_{2}dt),\hspace{0.2cm}r\in\mathcal{P}(\mathcal{Z}%
\times\mathcal{Z}%
\times\lbrack0,T])
\]
is a tightness function, i.e., it is bounded from below and its level sets
$R_{k}=\{r\in\mathcal{P}(\mathbb{R}^{2k}\times\lbrack
0,T]):g(r)\leq k\}$ are relatively compact for each $k<\infty$. Notice that the second marginal of every $r\in\mathcal{P}(\mathcal{Z}
\times\mathcal{Z}\times\lbrack0,T])$ is the Lebesgue measure.

Chebyshev's inequality implies
\[
\sup_{r\in R_{k}}r\left(  \{(z_{1},z_{2})\in\mathcal{Z}%
\times\mathcal{Z}:\left[\left\Vert
z_{1}\right\Vert+\left\Vert
z_{2}\right\Vert\right] >M\}\times\lbrack0,T]\right)  \leq\sup_{r\in R_{k}}\frac
{g(r)}{M^{2}}\leq\frac{k}{M^{2}}.
\]
Hence, $R_{k}$ is tight and thus relatively compact as a subset of
$\mathcal{P}$.

Since $g$ is a tightness function, by Theorem A.3.17 of \cite{DupuisEllis}
tightness of $\{\mathrm{P}^{\epsilon,\Delta,\gamma}_{1,t},\epsilon>0\}$ will follow if we
prove that
\[
\sup_{\epsilon\in(0,1]}\mathbb{E}\left[  g(\mathrm{P}^{\epsilon
,\Delta,\gamma}_{1,t}\otimes \text{Leb}_{[0,T]})\right]  <\infty,
\]
where $\text{Leb}_{[0,T]}$ denotes Lebesgue measure in $[0,T]$. However, by (\ref{Eq:Ubound})
\begin{align}
\sup_{\epsilon\in(0,1]}\mathbb{E}\left[  g(\mathrm{P}^{\epsilon
,\Delta,\gamma}_{1,t}\otimes \text{Leb}_{[0,T]})\right]   &  =\sup_{\epsilon\in(0,1]}\mathbb{E}\left[
\int_{0}^{T}\int_{\mathcal{Z}\times\mathcal{Z}}\left[\left\Vert z_{1}\right\Vert^{2}+\left\Vert z_{2}\right\Vert^{2}\right]\mathrm{P}^{\epsilon,\Delta,\gamma}_{1,t}(dz_{1}dz_{2})dt\right] \nonumber\\
&  =\sup_{\epsilon\in(0,1]}\mathbb{E}\int_{0}^{T}\frac{1}{\Delta}%
\int_{t}^{t+\Delta}\left[\left\Vert u_{1}^{\epsilon}(s)\right\Vert ^{2}+\left\Vert u_{2}^{\epsilon}(s)\right\Vert ^{2}\right]dsdt\nonumber\\
&  <\infty,\nonumber
\end{align}
uniformly in $\gamma\in\Gamma$, which concludes the tightness proof for $\{\mathrm{P}^{\epsilon,\Delta,\gamma}_{1,t},\epsilon>0\}$.

Let us now consider tightness of $\{\mathrm{P}^{\epsilon,\Delta,\gamma}_{2,t},\epsilon>0\}$. For this purpose we notice that for every
 $\gamma\in\Gamma$ and every $\tilde{\phi}\in L^{2}(\Gamma)\cap L^{1}(\pi)$ we have
\[
\int_{\Gamma}\tilde{\phi}(\tilde{\gamma})\mathrm{P}^{\epsilon,\Delta,\gamma}_{2,t}(d\tilde{\gamma})=\frac{1}{\Delta}\int_{t}^{t+\Delta}\tilde{\phi}
\left(\tau_{\bar{Y}^{\epsilon}_{s}}\gamma\right)ds=
\frac{1}{\Delta}\int_{t}^{t+\Delta}\phi\left(\bar{Y}^{\epsilon}_{s},\gamma\right)ds.
\]

Let us fix $\eta>0$. Then, by Lemma \ref{L:ErgodicTheorem4} we know that there exists $N_{\eta}\subset \Gamma$  with $\pi(N_{\eta})\geq 1-\eta$ such that for every bounded sequence $\Delta\in \mathcal{H}^{N_{\eta}}_{1}$ we have

\begin{equation*}
\lim_{\epsilon\downarrow0}\sup_{\gamma\in N_{\eta}}\sup_{0\leq t\leq T}\mathbb{E}\left\vert
\frac{1}{\Delta}\int_{t}^{t+\Delta}\phi\left(\bar{Y}_{s}^{\epsilon},\gamma\right)ds-\bar{\phi}\right\vert =0
\end{equation*}
 or equivalently
\begin{equation}
\lim_{\epsilon\downarrow0}\sup_{\gamma\in N_{\eta}}\sup_{0\leq t\leq T}\mathbb{E}\left\vert
\int_{\Gamma}\tilde{\phi}(\tilde{\gamma})\mathrm{P}^{\epsilon,\Delta,\gamma}_{2,t}(d\tilde{\gamma})-\bar{\phi}\right\vert =0\label{Eq:UsingErgodicThmForTightness}
\end{equation}
 Now, as a probability measure in a Polish space $\pi$ is itself tight. So, there exists a compact subset of $\Gamma$, say $K_{\eta}$, such that
\[
\pi(K_{\eta})\geq 1-\eta/2.
\]


Therefore, using (\ref{Eq:UsingErgodicThmForTightness}) and the latter bound, we get that for $\epsilon$ sufficiently small, say $\epsilon<\epsilon_{0}(\eta)$ and for every $\gamma\in N_{\eta}$ and $t\in[0,T]$, we have

\begin{equation*}
\inf_{\epsilon\in(0,\epsilon_{0}(\eta))}\mathbb{E}\left[\mathrm{P}^{\epsilon,\Delta,\gamma}_{2,t}(K_{\eta})\right]\geq 1-\eta
\end{equation*}
which implies that, uniformly in $\gamma\in N_{\eta}$, the measure valued random variables $\{\mathrm{P}^{\epsilon,\Delta,\gamma}_{2,t}(\cdot),\epsilon\in(0,\epsilon_{0}(\eta))\}$ are tight.

(ii). Uniform integrability of the family  $\{\mathrm{P}^{\epsilon,\Delta,\gamma},\epsilon>0\}$ follows by
\begin{align}
&\mathbb{E}\left[  \int_{\{(z_{1},z_{2})\in\mathcal{Z}\times\mathcal{Z}:\left[\left\Vert z_{1}\right\Vert+\left\Vert z_{2}\right\Vert\right]
>M\}\times\Gamma\times\lbrack0,T]}\left[\left\Vert z_{1}\right\Vert +\left\Vert z_{2}\right\Vert\right]\mathrm{P}%
^{\epsilon,\Delta}(dz_{1}dz_{2}d\tilde{\gamma}dt)\right]  \nonumber\\
&\quad\leq\frac{2}{M}\mathbb{E}\left[
\int_{\mathcal{Z}\times\mathcal{Z}\times\Gamma\times\lbrack0,T]}\left[\left\Vert z_{1}\right\Vert
^{2}+\left\Vert z_{2}\right\Vert
^{2}\right]\mathrm{P}^{\epsilon,\Delta}(dz_{1}dz_{2}d\tilde{\gamma}dt)\right]\nonumber\\
& \quad = \frac{2}{M} \mathbb{E}\int_{0}^{T}\frac{1}{\Delta}%
\int_{t}^{t+\Delta}\left[\left\Vert u_{1}^{\epsilon}(s)\right\Vert ^{2}+\left\Vert u_{2}^{\epsilon}(s)\right\Vert ^{2}\right]dsdt\nonumber
\end{align}
and the fact that
\[
\sup_{\epsilon>0,\gamma\in\Gamma}\mathbb{E}\int_{0}^{T}\frac{1}{\Delta}%
\int_{t}^{t+\Delta}\left[\left\Vert u_{1}^{\epsilon}(s)\right\Vert ^{2}+\left\Vert u_{2}^{\epsilon}(s)\right\Vert ^{2}\right]dsdt  <\infty.
\]
This concludes the proof of the lemma.
\end{proof}

\begin{lemma}
\label{L:TightnessControlledProcess} Assume Conditions \ref{A:Assumption1} and \ref{A:Assumption2}.
Let $\{u^{\epsilon,\gamma},\epsilon>0,\gamma\in\Gamma\}$ be a family of
controls in $\mathcal{A}$ as in Lemma \ref{L:TightnessOccupationalMeasures}. Moreover, fix $\eta>0$, and consider the set $N_{\eta}$ with $\pi(N_{\eta})\geq 1-\eta$ from Lemma \ref{L:ErgodicTheorem4}. Then,  for every $\gamma\in N_{\eta}$, the family $\{\bar{X}^{\epsilon,\gamma},\epsilon>0\}$ is
relatively compact as $\epsilon\downarrow 0$.

\end{lemma}
\begin{proof}
It suffices to prove  that for every $\eta>0$
\[
\lim_{\theta\downarrow0}\limsup_{\epsilon\downarrow0}\mathbb{P}\left[  \sup_{t_{1},t_{2}<T, |t_{1}-t_{2}|<\theta}\left\Vert  \bar{X}%
^{\epsilon,\gamma}_{t_{1}}-\bar{X}^{\epsilon,\gamma}_{t_{2}}\right\Vert
>\eta\right]  =0
\]

Recalling the auxiliary problem (\ref{Eq:RandomCellProblem}) and the discussion succeeding it, we apply
It\^{o} formula (see also \cite{Osada1983}), to rewrite $\bar{X}^{\epsilon,\gamma}_{t_{1}}-\bar
{X}^{\epsilon,\gamma}_{t_{2}}$ as

\begin{eqnarray}
\bar{X}^{\epsilon,\gamma}_{t_{1}}-\bar{X}^{\epsilon,\gamma}_{t_{2}} &=&  \int_{t_{1}}^{t_{2}} \lambda\left(  \bar{X}^{\epsilon,\gamma}_{s},\bar{Y}^{\epsilon,\gamma}_{s}, u_{1}(s),u_{2}(s)\right) ds\nonumber\\
& &\quad -\delta\left[
\chi_{0}\left(  \bar{Y}^{\epsilon,\gamma}_{t_{2}}\right)
-\chi_{0}\left(  \bar{Y}^{\epsilon,\gamma}_{t_{1}}
\right)\right] \nonumber\\
& &\quad+
\sqrt{\epsilon}\int_{t_{1}}^{t_{2}} \left(\sigma+\xi\tau_{1}\right)\left(  \bar{X}^{\epsilon,\gamma}_{s}, \bar{Y}^{\epsilon,\gamma}_{s}\right)
dW_{s}+\sqrt{\epsilon}\int_{t_{1}}^{t_{2}}\xi\tau_{2}\left(  \bar{Y}^{\epsilon,\gamma}_{s}\right)dB_{s}\nonumber\\
&  &=B^{\epsilon,\gamma}_{1}+ B^{\epsilon,\gamma}_{2}+ B^{\epsilon,\gamma}_{3}\nonumber
\end{eqnarray}
where $B^{\epsilon,\gamma}_{i}$ is the $i^{th}$ line of the right hand side of
 the last display.

First we treat the term $B^{\epsilon,\gamma}_{3}$. It suffices to discuss one of the two stochastic integrals, let's say the first one. In
particular, by It\^{o} isometry, Lemma \ref{L:ErgodicTheorem4}, we have, that  there is a set $N_{\eta}$ with $\pi(N_{\eta})\geq 1-\eta$ such that for every $\gamma\in N_{\eta}$,
\[
\lim_{\epsilon\downarrow0}\left|  \mathbb{E}\left\Vert  \int_{t_{1}}^{t_{2}}\left(\tilde{\sigma}\left( \bar{X}^{\epsilon,\gamma}_{s}, \cdot\right)+\tilde{\xi}\tilde{\tau}_{1}\left(\cdot\right)\right)dW_{s}  \right\Vert  ^{2}-
\int_{t_{1}}^{t_{2}}\mathrm{E}^{\pi}\left[  \left\Vert  \left(
\tilde{\sigma}\left( \bar{X}^{\epsilon,\gamma}_{s}, \cdot\right)+\tilde{\xi}\tilde{\tau}_{1}\left(\cdot\right) \right)   \right\Vert  ^{2}\right]ds  \right|  \rightarrow0
\]
as $\epsilon\downarrow0$. In a similar fashion we can also treat the stochastic integral with respect to the Brownian motion $B$. Hence, for every
$\gamma\in N_{\eta}$
\[
\lim_{\epsilon\downarrow0}\mathbb{E}\left\Vert B^{\epsilon,\gamma}_{3}\right\Vert^{2}=0
\]

Next, we treat $B^{\epsilon,\gamma}_{1}$. Lemma \ref{L:ErgodicTheorem4} and the uniform bound (\ref{Eq:UniformlySquareIntegrableControlsAdditional}), implies that for every
$\gamma\in N_{\eta}$
\[
\lim_{|t_{2}-t_{1}|\rightarrow0 }\lim_{\epsilon\downarrow0} \mathbb{E}%
\left\Vert B^{\epsilon,\gamma,t_{2}-t_{1}}_{1}\right\Vert^{2}=0
\]

Similarly, one can show that $\lim_{\epsilon\downarrow0}\mathbb{E}\left\Vert  B^{\epsilon,\gamma}_{2}
\right\Vert  =0$. Therefore, tightness of $\{\bar{X}^{\epsilon,\gamma},\epsilon>0\}$ follows for $\gamma\in N_{\eta}$.
\end{proof}

\subsection{Identification of the limit points.} \label{S:WeakConvergence}
In this section we prove that any weak limit point of the tight sequence
$\left\{  (\bar{X}^{\epsilon,\gamma},\mathrm{P}^{\epsilon,\Delta,\gamma}),
\epsilon>0 \right\}  $ is a viable pair, i.e., it satisfies Definition
\ref{D:ViablePair}. Let $(\bar{X},\mathrm{P})$ be an accumulation point (in
distribution) of $(\bar{X}^{\epsilon,\gamma},\mathrm{P}^{\epsilon
,\Delta,\gamma})$ as $\epsilon,\Delta\downarrow0$. Due to the Skorokhod
representation, we may assume that there is a probability space, where this
convergence holds with probability $1$. The constraint
(\ref{Eq:Ubound}) and Fatou's lemma guarantee that with
probability $1$,
\[
\int_{\mathcal{Z}\times\mathcal{Z}\times\Gamma\times[0,T]}\left[\left\Vert z_{1}\right\Vert^{2}+\left\Vert z_{2}\right\Vert^{2}\right]\bar{\mathrm{P}}(dz_{1}dz_{2} d\gamma
dt)<\infty.
\]

Moreover, since $\mathrm{P}^{\epsilon,\Delta,\gamma}\left(  \mathcal{Z}%
\times\mathcal{Z}\times\Gamma\times[0,t]\right)  =t$ for every $t\in[0,T]$ and using the fact
that $\bar{\mathrm{P}}\left(  \mathcal{Z}\times\mathcal{Z}\times\Gamma\times[0,t]\right)  $ is
continuous as a function of $t\in[0,T]$ and that $\bar{\mathrm{P}}\left(
\mathcal{Z}\times\mathcal{Z}\times\Gamma\times\left\{  t\right\}  \right)  =0$ we obtain
$\bar{\mathrm{P}}\left(  \mathcal{Z}\times\mathcal{Z}\times\Gamma\times[0,t]\right)  =t$ and
that $\bar{\mathrm{P}}$ can be decomposed as $\mathrm{P}(dz_{1}dz_{2} d\gamma
dt)=\mathrm{P}_{t}(dz_{1}dz_{2} d\gamma)dt$ with $\mathrm{P}_{t}(\mathcal{Z}\times\mathcal{Z}\times
\Gamma)=1$.

Let us next prove that $(\bar{X},\bar{\mathrm{P}})$ satisfy
(\ref{Eq:ViablePair1}). We will use the martingale problem. In particular, let
$\zeta$ be a smooth bounded function, $\phi\in\mathcal{C}^{2}(\mathbb{R}^{m})$
compactly supported, $\left\{  \tilde{z}_{j}\right\}  _{j=1}^{q}$ be a
family of bounded, smooth and compactly supported functions and for
$r\in\mathcal{P}\left(  \mathcal{Z}\times\mathcal{Z}\times\Gamma\times[0,T]\right)  $,
$t\in[0,T]$ define
\[
\left(  r,\tilde{z}_{j}\right)  _{t}=\int_{\mathcal{Z}\times\mathcal{Z}\times\Gamma
\times[0,t]} \tilde{z}_{j}(z_{1},z_{2},\gamma,s)r(dz_{1}dz_{2}d\gamma ds)
\]

Then, in order to show (\ref{Eq:ViablePair1}), it is enough to show that for
any $0<t_{1}<t_{2}<\cdots<t_{m}<t<t+r\leq T$, the following limit holds almost
surely with respect to $\gamma\in\Gamma$ as $\epsilon\downarrow0$
\begin{align}
\mathbb{E}  &  \left\{  \zeta\left(  \bar{X}^{\epsilon,\gamma}%
_{t_{i}},(\mathrm{P}^{\epsilon,\Delta,\gamma},z_{j})_{t_{i}}, i\leq m,
j\leq q\right)  \left[  \phi(\bar{X}^{\epsilon,\gamma}_{t+r})-\phi(\bar{X}%
^{\epsilon,\gamma}_{t})\right.  \right. \nonumber\\
&  \left.  \left.  \hspace{2cm}-\int_{t}^{t+r}\left[  \lim_{\rho\rightarrow
0}\int_{\mathcal{Z}\times\mathcal{Z}\times\Gamma}\tilde{\lambda}_{\rho}(\bar{X}^{\epsilon
,\gamma}_{s},\gamma,z_{1},z_{2})\mathrm{P}_{s}(dz_{1}dz_{2}d\gamma) \right]  \nabla\bar{\phi}(\bar
{X}^{\epsilon,\gamma}_{s})ds\right]  \right\}  \rightarrow0
\label{Eq:TargetForConvergenceControlled1}%
\end{align}

Let us define
\[
\mathcal{L}^{\epsilon,\Delta,\rho}_{s}\phi(x)=\int_{\mathcal{Z}\times\mathcal{Z}\times\Gamma
}\tilde{\lambda}_{\rho}(x,\gamma,z_{1},z_{2})\mathrm{P}^{\epsilon,\Delta,\gamma}_{s}(dz_{1}dz_{2}
d\gamma) \nabla \phi(x)
\]
where
\[
\mathrm{P}^{\epsilon,\Delta,\gamma}_{s}(dz_{1}dz_{2} d\gamma) =\frac{1}{\Delta}\int
_{s}^{s+\Delta}1_{z_{1}}(u^{\epsilon}_{1}(\theta))1_{z_{2}}(u^{\epsilon}_{2}(\theta))1_{B}\left(  \tau_{\bar
{Y}^{\epsilon,\gamma}_{\theta}}\gamma\right)  d\theta
\]

Then, weak convergence of the pair $(\bar{X}^{\epsilon,\gamma},\mathrm{P}%
^{\epsilon,\Delta,\gamma})$ and uniform integrability of $\mathrm{P}%
^{\epsilon,\Delta,\gamma}$ as indicated by Lemma
\ref{L:TightnessOccupationalMeasures}, shows that almost surely with respect
to $\gamma\in\Gamma$
\begin{align}
\mathbb{E}  &  \left[  \int_{t}^{t+r}\mathcal{L}^{\epsilon
,\Delta,\rho}_{s}\phi(\bar{X}^{\epsilon,\gamma}_{s})ds-\int_{t}^{t+r}\left[
\lim_{\rho\rightarrow0}\int_{\mathcal{Z}\times\mathcal{Z}\times\Gamma}\tilde{\lambda}_{\rho
}(\bar{X}^{\epsilon,\gamma}_{s},\gamma,z_{1},z_{2})\mathrm{P}_{s}(dz_{1}dz_{2}d\gamma) \right]
\nabla \phi(\bar{X}^{\epsilon,\gamma}_{s})ds\right]  \rightarrow0\nonumber
\end{align}
as $\epsilon\downarrow0$ and $\rho=\rho(\epsilon)\downarrow 0$. Hence, in order to prove
(\ref{Eq:TargetForConvergenceControlled1}), it is sufficient to prove that
almost surely with respect to $\gamma\in\Gamma$
\[
\mathbb{E} \left\{  \zeta\left(  \bar{X}^{\epsilon,\gamma}_{t_{i}%
},(\mathrm{P}^{\epsilon,\Delta,\gamma},z_{j})_{t_{i}}, i\leq m, j\leq
q\right)  \left[  \phi(\bar{X}^{\epsilon,\gamma}_{t+r})-\phi(\bar{X}^{\epsilon
,\gamma}_{t})-\int_{t}^{t+r}\mathcal{L}^{\epsilon,\Delta,\rho}_{s}\phi(\bar
{X}^{\epsilon,\gamma}_{s})ds\right]  \right\}  \rightarrow0
\]

Recall the auxiliary problem (\ref{Eq:RandomCellProblem}) and consider a
function $\phi\in\mathcal{C}^{2}(\mathbb{R}^{m})$ with compact support. Let us
write $\chi_{\rho}=\left(  \chi_{1,\rho},\ldots,\chi_{m,\rho}\right)  $ for
the components of the vector solution to (\ref{Eq:RandomCellProblem}), and
consider $\psi_{\ell,\rho}(x,y,\gamma)=\chi_{\ell,\rho}(y,\gamma
)\partial_{x_{\ell}}\phi(x)$ for $\ell\in\{1,\ldots,m\}$. Set $\psi_{\rho
}(x,y,\gamma)=\left(  \psi_{1,\rho},\ldots,\psi_{m,\rho}\right)  $. It is easy
to see that $\tilde{\psi}_{\rho}(x,\gamma)$ satisfies the resolvent equation%

\begin{equation}
\rho\tilde{\psi}_{\ell,\rho}(x,\cdot)-\tilde{L}\tilde{\psi}_{\ell,\rho
}(x,\cdot)=\tilde{h}_{\ell}(x,\cdot)
\label{Eq:RandomCellProblem1}%
\end{equation}
where we have defined $\tilde{h}_{\ell}(x,\cdot)=\tilde{b}_{\ell}%
(\cdot)\partial_{x_{\ell}}\phi(x)$. By It\^{o} formula and making use of (\ref{Eq:RandomCellProblem1}), we obtain%

\begin{align}
&  \mathbb{E}\left\{  \zeta\left(  \bar{X}^{\epsilon,\gamma}_{t_{i}%
},(\mathrm{P}^{\epsilon,\Delta,\gamma},z_{j})_{t_{i}}, i\leq m, j\leq
q\right)  \left[  \phi(\bar{X}^{\epsilon,\gamma}_{t+r})-\phi(\bar{X}^{\epsilon
,\gamma}_{t})-\int_{t}^{t+r}\mathcal{L}^{\epsilon,\Delta,\rho}_{s}\phi(\bar
{X}^{\epsilon,\gamma}_{s})ds\right]  \right\} \nonumber\\
&  = \mathbb{E}\left\{  \zeta\left(  \cdots\right)  \left[  \int
_{t}^{t+r}\lambda_{\rho}\left(  \bar{X}^{\epsilon,\gamma}_{s},\bar
{Y}^{\epsilon,\gamma}_{s},\gamma,u^{\epsilon}_{1}(s),u^{\epsilon}_{2}(s)\right)  \nabla
\phi(\bar{X}^{\epsilon,\gamma}_{s})ds- \int_{t}^{t+r}\mathcal{L}^{\epsilon
,\Delta,\rho}_{s}\phi(\bar{X}^{\epsilon,\gamma}_{s})ds\right]  \right\}
\nonumber\\
&  \hspace{0.2cm}+\delta\mathbb{E}\left\{  \zeta\left(
\cdots\right)  \int_{t}^{t+r}\sum_{\ell=1}^{m}\left(  (c+\sigma
u^{\epsilon}_{1}(s)) \partial_{x} \psi_{\ell,\rho} +\epsilon\frac{1}{2}
\text{tr}\left[  \partial^{2}_{x}\psi_{\ell,\rho}\right]  \right)  \left(
\bar{X}^{\epsilon,\gamma}_{s},\bar{Y}^{\epsilon,\gamma}_{s}\right)  \right\}  ds\nonumber\\
&  \hspace{0.2cm}+\epsilon\mathbb{E}\left\{  \zeta\left(
\cdots\right)  \int_{t}^{t+r}\sum_{\ell=1}^{m}\text{tr}\left[\sigma\tau_{1}^{T}  D\partial
_{x}\psi_{\ell,\rho}\right]  \left(  \bar{X}^{\epsilon,\gamma}_{s},\bar
{Y}^{\epsilon,\gamma}_{s}\right)  \right\}  ds\nonumber\\
&  \hspace{0.2cm}+\epsilon2\mathbb{E}\left\{  \zeta\left(
\cdots\right)  \int_{t}^{t+r} \text{tr}\left[  \sigma\sigma^{T}\left(
\bar{Y}^{\epsilon,\gamma}_{s}\right)  \nabla^{2}\phi\left(
\bar{X}^{\epsilon,\gamma}_{s}\right)  \right]  \right\}  ds\nonumber\\
&  \hspace{0.2cm} +\frac{\epsilon}{\delta}\rho\mathbb{E}\left\{
\zeta\left(  \cdots\right)  \int_{t}^{t+r}\chi_{\rho}\left(  \bar
{Y}^{\epsilon,\gamma}_{s}\right)  \nabla \phi(\bar{X}^{\epsilon,\gamma
}_{s})ds\right\} \nonumber\\
&  \hspace{0.2cm}-\delta\sum_{\ell=1}^{m}\mathbb{E}\left\{
\zeta\left(  \cdots\right)  \left(  \psi_{\ell,\rho}\left(  \bar{X}%
^{\epsilon,\gamma}_{t+r},\bar
{Y}^{\epsilon,\gamma}_{t+r}\right)  -\psi_{\ell,\rho}\left(  \bar{X}^{\epsilon,\gamma}_{t},\bar
{Y}^{\epsilon,\gamma}_{t}\right)  \right)  \right\} \nonumber\\
&  = \sum_{i=1}^{6}\mathbb{E}B_{i}^{\epsilon,\gamma}
\label{Eq:TargetForConvergenceControlled2}%
\end{align}
where $\mathbb{E}B_{i}^{\epsilon,\gamma}$ is the $i^{\text{th}}$ line
on the right hand side of (\ref{Eq:TargetForConvergenceControlled2}). We want
to show that each of those terms goes to zero almost surely with respect to
$\gamma\in\Gamma$.

Condition \ref{A:Assumption1} and the bound (\ref{Eq:Ubound}) give us that
\[
\mathbb{E}\left| B_{2}^{\epsilon,\gamma}\right| +\mathbb{E}\left| B_{3}^{\epsilon,\gamma}\right|  \rightarrow0, \quad\text{ as }%
\epsilon\downarrow0
\]

Due to the boundedness and compact support of functions $\zeta$ and $\phi$, we
also get that almost surely in $\gamma\in\Gamma$
\[
\mathbb{E}\left|  B_{4}^{\epsilon,\gamma}\right|  \rightarrow0,
\quad\text{ as }\epsilon\downarrow0
\]

By choosing $\rho=\rho(\epsilon)=\frac{\delta^{2}%
}{\epsilon}$, we also have that almost surely in $\gamma\in\Gamma$
\[
\mathbb{E}\left|  B_{5}^{\epsilon,\gamma}\right| +\mathbb{E}\left|  B_{6}^{\epsilon,\gamma}\right| \rightarrow0,
\quad\text{ as }\epsilon\downarrow0
\]

Let us next consider $B_{1}^{\epsilon,\gamma}$. We have
\begin{align}
  \mathbb{E}B_{1}^{\epsilon,\gamma}
&  =\mathbb{E}\left\{  \zeta\left( \cdots\right)  \left[  \int_{t}^{t+r}\lambda_{\rho}\left(  \bar{X}^{\epsilon
,\gamma}_{s},\bar{Y}^{\epsilon,\gamma}_{s},\gamma,u^{\epsilon}_{1}(s),u^{\epsilon}_{2}(s)\right) \nabla \phi(\bar{X}^{\epsilon,\gamma}_{s}) ds- \int_{t}^{t+r}\mathcal{L}^{\epsilon,\Delta,\rho}%
_{s}\phi(X^{\epsilon,\gamma}_{s})ds\right]  \right\} \nonumber\\
&  = \mathbb{E}\left\{  \zeta\left(  \cdots\right)  \left[  \int
_{t}^{t+r}\lambda_{\rho}\left(  \bar{X}^{\epsilon,\gamma}_{s},\bar
{Y}^{\epsilon,\gamma}_{s},\gamma,u^{\epsilon}_{1}(s),u^{\epsilon}_{2}(s)\right) \nabla \phi(\bar{X}^{\epsilon,\gamma}_{s}) ds-\right.\right.\nonumber\\
&\hspace{3cm}\left.\left.-\int_{t}^{t+r}\frac{1}{\Delta}\int_{s}^{s+\Delta} \lambda_{\rho}\left(
\bar{X}^{\epsilon,\gamma}_{s},\bar{Y}^{\epsilon,\gamma}_{\theta},\gamma,u^{\epsilon}_{1}(\theta),
u^{\epsilon}_{2}(\theta)\right) \nabla \phi(\bar{X}^{\epsilon,\gamma}_{s}) d\theta ds\right]  \right\}
\nonumber\\
&  = \mathbb{E}\left\{  \zeta\left(  \cdots\right)  \left[  \int
_{t}^{t+r}\frac{1}{\Delta}\int_{s}^{s+\Delta}\lambda_{\rho}\left(  \bar
{X}^{\epsilon,\gamma}_{\theta},\bar{Y}^{\epsilon,\gamma}_{\theta}%
,\gamma,u^{\epsilon}_{1}(\theta),u^{\epsilon}_{2}(\theta)\right) \nabla \phi(\bar{X}^{\epsilon,\gamma}_{\theta}) d\theta ds-\right.  \right.
\nonumber\\
&  \hspace{3cm}\left.  \left.  -\int_{t}^{t+r}\frac{1}{\Delta}\int
_{s}^{s+\Delta} \lambda_{\rho}\left(  \bar{X}^{\epsilon,\gamma}_{s},\bar{Y}^{\epsilon,\gamma}_{\theta},\gamma,u^{\epsilon}_{1}(\theta),u^{\epsilon}_{2}(\theta)
\right) \nabla \phi(\bar{X}^{\epsilon,\gamma}_{s})  d\theta ds\right]  \right\} \nonumber\\
&  + \mathbb{E}\left\{  \zeta\left(  \cdots\right)  \left[  \int
_{t}^{t+r}\lambda_{\rho}\left(  \bar{X}^{\epsilon,\gamma}_{s},\bar
{Y}^{\epsilon,\gamma}_{s},\gamma,u^{\epsilon}_{1}(s),u^{\epsilon}_{2}(s)\right) \nabla \phi(\bar{X}^{\epsilon,\gamma}_{s}) ds-\right.\right.\nonumber\\
&\hspace{3cm}\left.\left.-\int_{t}^{t+r}\frac{1}{\Delta}\int_{s}^{s+\Delta} \lambda_{\rho}\left(
\bar{X}^{\epsilon,\gamma}_{\theta},\bar{Y}^{\epsilon,\gamma}_{\theta}%
,\gamma,u^{\epsilon}_{1}(\theta),u^{\epsilon}_{2}(\theta)\right) \nabla \phi(\bar{X}^{\epsilon,\gamma}_{\theta}) d\theta ds\right]  \right\}
\nonumber\\
&  = \mathbb{E}B_{1,1}^{\epsilon,\gamma}+\mathbb{E}%
B_{1,2}^{\epsilon,\gamma}\nonumber
\end{align}
Let us first treat $\mathbb{E}B_{1,1}^{\epsilon,\gamma}$.%

\begin{align}
&  \mathbb{E}B_{1,1}^{\epsilon,\gamma}=\nonumber\\
&  = \mathbb{E}\left\{  \zeta\left(  \cdots\right)  \left[  \int
_{t}^{t+r}\frac{1}{\Delta}\int_{s}^{s+\Delta}\lambda_{\rho}\left(  \bar
{X}^{\epsilon,\gamma}_{\theta},\bar{Y}^{\epsilon,\gamma}_{\theta}%
,\gamma,u^{\epsilon}_{1}(\theta),u^{\epsilon}_{2}(\theta)\right) \nabla \phi(\bar{X}^{\epsilon,\gamma}_{\theta})  d\theta ds-\right.  \right.
\nonumber\\
&  \hspace{6cm}\left.  \left.  - \int_{t}^{t+r}\frac{1}{\Delta}\int
_{s}^{s+\Delta} \lambda_{\rho}\left(  \bar{X}^{\epsilon,\gamma}_{s},\bar{Y}^{\epsilon,\gamma}_{\theta},\gamma,u^{\epsilon}_{1}(\theta),u^{\epsilon}_{2}(\theta)
\right) \nabla \phi(\bar{X}^{\epsilon,\gamma}_{s}) d\theta ds\right]  \right\} \nonumber\\
&  = \mathbb{E}\left\{  \zeta\left(  \cdots\right)  \left[  \int
_{t}^{t+r}\frac{1}{\Delta}\int_{s}^{s+\Delta}\tilde{\lambda}_{\rho}\left(
\bar{X}^{\epsilon,\gamma}_{\theta},\tau_{\bar{Y}^{\epsilon,\gamma
}_{\theta}}\gamma,u^{\epsilon}_{1}(\theta),u^{\epsilon}_{2}(\theta)\right) \nabla \phi(\bar{X}^{\epsilon,\gamma}_{\theta}) d\theta ds-\right.
\right. \nonumber\\
&  \hspace{6cm}\left.  \left.  -\int_{t}^{t+r}\frac{1}{\Delta}\int
_{s}^{s+\Delta} \tilde{\lambda}_{\rho}\left(  \bar{X}^{\epsilon,\gamma}%
_{s},\tau_{\bar{Y}^{\epsilon,\gamma}_{\theta}}\gamma
,u^{\epsilon}_{1}(\theta),u^{\epsilon}_{2}(\theta)\right)\nabla \phi(\bar{X}^{\epsilon,\gamma}_{s})  d\theta ds\right]  \right\} \nonumber\\
&  \rightarrow0, \quad\text{ as }\epsilon\downarrow0,\nonumber
\end{align}
by continuity of $\tilde{\lambda}_{\rho}$ on the first argument, stationarity
and the uniform integrability obtained in Lemma
\ref{L:TightnessOccupationalMeasures}.

Next we treat $\mathbb{E}B_{1,2}^{\epsilon,\gamma}$. We have%

\begin{align}
\mathbb{E} \left|  B_{1,2}^{\epsilon,\gamma}\right|   &  \leq
C_{0}\left\{  \mathbb{E} \int_{0}^{\Delta}\left|  \lambda_{\rho
}\left(  \bar{X}^{\epsilon,\gamma}_{s},\bar{Y}^{\epsilon,\gamma}_{s}%
,\gamma,u^{\epsilon}_{1}(s),u^{\epsilon}_{2}(s)\right) \nabla \phi(\bar{X}^{\epsilon,\gamma}_{s}) \right|  ds\right.\nonumber\\
&\hspace{3cm}\left.+ \mathbb{E}
\int_{t}^{t+\Delta}\left|  \lambda_{\rho}\left(  \bar{X}^{\epsilon,\gamma}%
_{s},\bar{Y}^{\epsilon,\gamma}_{s},\gamma,u^{\epsilon}_{1}(s),u^{\epsilon}_{2}(s)\right) \nabla \phi(\bar{X}^{\epsilon,\gamma}_{s}) \right|  ds\right\} \nonumber
\end{align}

where $C_{0}$ is a finite constant.  Choose $\Delta\downarrow0$ such that $\Delta
/\frac{\delta^{2}}{\varepsilon}\uparrow\infty$. Then, we have
{\small\begin{align}
&  \mathbb{E} \int_{0}^{\Delta}\left|  \lambda_{\rho}\left(  \bar
{X}^{\epsilon,\gamma}_{s},\bar{Y}^{\epsilon,\gamma}_{s}
,\gamma,u^{\epsilon}_{1}(s),u^{\epsilon}_{2}(s)\right) \nabla \phi(\bar{X}^{\epsilon,\gamma}_{s}) \right|  ds \nonumber\\
&\leq\mathbb{E} \int
_{0}^{\Delta}\left| \left( c\left(  \bar
{X}^{\epsilon,\gamma}_{s},\bar{Y}^{\epsilon,\gamma}_{s}
,\gamma\right)+D \chi_{\rho}\left( \bar{Y}%
^{\epsilon,\gamma}_{s},\gamma\right)g\left(  \bar
{X}^{\epsilon,\gamma}_{s},\bar{Y}^{\epsilon,\gamma}_{s}
,\gamma\right)    \right)\nabla \phi(\bar{X}^{\epsilon,\gamma}_{s})  \right|  ds\nonumber\\
&  + \mathbb{E} \int_{0}^{\Delta}\left|\left(
\sigma\left( \bar{X}^{\epsilon,\gamma}_{s},  \bar{Y}^{\epsilon,\gamma}_{s},\gamma\right)  u^{\epsilon}_{1}(s)+D
\chi_{\rho}\left(  \bar{Y}^{\epsilon,\gamma}_{s},\gamma\right)
\left[\tau_{1}\left(  \bar{Y}^{\epsilon,\gamma}_{s},\gamma\right)  u^{\epsilon}_{1}(s)+
\tau_{2}\left(  \bar{Y}^{\epsilon,\gamma}_{s},\gamma\right)  u^{\epsilon}_{2}(s)\right]\right)\nabla \phi(\bar{X}^{\epsilon,\gamma}_{s})\right|  ds\nonumber\\
& \leq\Delta\frac{\frac{\delta^{2}}{\epsilon}}{\Delta}
\mathbb{E} \int_{0}^{\Delta/\frac{\delta^{2}}{\epsilon}}\left|\left(
 c\left(  \bar{X}^{\epsilon,\gamma}_{(\delta^{2}/\epsilon)s},\bar{Y}^{\epsilon,\gamma}_{(\delta
^{2}/\epsilon)s},\gamma\right)+D \chi_{\rho}\left( \bar{Y}^{\epsilon,\gamma}_{(\delta^{2}/\epsilon)s
},\gamma\right)g\left(  \bar{X}^{\epsilon,\gamma}_{(\delta^{2}/\epsilon)s},\bar{Y}^{\epsilon,\gamma}_{(\delta
^{2}/\epsilon)s},\gamma\right)  \right)\nabla \phi(\bar{X}^{\epsilon,\gamma}_{(\delta^{2}/\epsilon)s})\right|  ds\nonumber\\
&  + \sqrt{\Delta}\sqrt{\frac{\frac{\delta^{2}}{\epsilon}}{\Delta
}\mathbb{E} \int_{0}^{\Delta/\frac{\delta^{2}}{\epsilon}}\left\Vert
\left(\sigma\left( \bar{X}^{\epsilon,\gamma}_{(\delta^{2}/\epsilon)s},  \bar{Y}^{\epsilon,\gamma}_{(\delta^{2}/\epsilon)s},\gamma\right)  +D
\chi_{\rho}\left(  \bar{Y}^{\epsilon,\gamma}_{(\delta^{2}/\epsilon)s},\gamma\right)
\tau_{1}\left(  \bar{Y}^{\epsilon,\gamma}_{(\delta^{2}/\epsilon)s},\gamma\right) \right)\nabla \phi \right\Vert  ^{2}ds \mathbb{E} \int_{0}%
^{\Delta}\left\Vert
u^{\epsilon}_{1}(s)\right\Vert  ^{2} ds }\nonumber\\
&  + \sqrt{\Delta}\sqrt{\frac{\frac{\delta^{2}}{\epsilon}}{\Delta
}\mathbb{E} \int_{0}^{\Delta/\frac{\delta^{2}}{\epsilon}}\left\Vert\left(
D
\chi_{\rho}\left(  \bar{Y}^{\epsilon,\gamma}_{(\delta^{2}/\epsilon)s},\gamma\right)\tau_{2}\left(  \bar{Y}^{\epsilon,\gamma}_{(\delta^{2}/\epsilon)s},\gamma\right)\right)\nabla \phi
 \right\Vert  ^{2}ds \mathbb{E} \int_{0}%
^{\Delta}\left\Vert
u^{\epsilon}_{2}(s)\right\Vert  ^{2} ds }\nonumber\\
&  \leq\Delta\frac{\frac{\delta^{2}}{\epsilon}}{\Delta}
\mathbb{E} \int_{0}^{\Delta/\frac{\delta^{2}}{\epsilon}}\left\Vert\left(
 \tilde{c}\left(  \bar{X}^{\epsilon,\gamma}_{(\delta^{2}/\epsilon)s},\tau_{\bar{Y}^{\epsilon,\gamma}_{(\delta^{2}/\epsilon)s}}\gamma\right)+D \tilde{\chi}_{\rho}\left( \tau_{\bar{Y}^{\epsilon,\gamma}_{(\delta^{2}/\epsilon)s}}\gamma\right)\tilde{g}\left(  \bar{X}^{\epsilon,\gamma}_{(\delta^{2}/\epsilon)s},\tau_{\bar{Y}^{\epsilon,\gamma}_{(\delta^{2}/\epsilon)s}}\gamma\right)  \right)\nabla \phi \right\Vert  ds\nonumber\\
&  + \sqrt{\Delta}\sqrt{\frac{\frac{\delta^{2}}{\epsilon}}{\Delta
}\mathbb{E} \int_{0}^{\Delta/\frac{\delta^{2}}{\epsilon}}\left\Vert\left(
\tilde{\sigma}\left( \bar{X}^{\epsilon,\gamma}_{(\delta^{2}/\epsilon)s},  \tau_{\bar{Y}^{\epsilon,\gamma}_{(\delta^{2}/\epsilon)s}}\gamma\right)  +D
\tilde{\chi}_{\rho}\left(  \tau_{\bar{Y}^{\epsilon,\gamma}_{(\delta^{2}/\epsilon)s}}\gamma\right)
\tilde{\tau}_{1}\left(  \tau_{\bar{Y}^{\epsilon,\gamma}_{(\delta^{2}/\epsilon)s}}\gamma\right)\right)\nabla \phi  \right\Vert  ^{2}ds \mathbb{E} \int_{0}%
^{\Delta}\left\Vert
u^{\epsilon}_{1}(s)\right\Vert^{2} ds }\nonumber\\
&  + \sqrt{\Delta}\sqrt{\frac{\frac{\delta^{2}}{\epsilon}}{\Delta
}\mathbb{E} \int_{0}^{\Delta/\frac{\delta^{2}}{\epsilon}}\left\Vert\left(
D
\tilde{\chi}_{\rho}\left(  \tau_{\bar{Y}^{\epsilon,\gamma}_{(\delta^{2}/\epsilon)s}}\gamma\right)
\tilde{\tau}_{2}\left(\tau_{\bar{Y}^{\epsilon,\gamma}_{(\delta^{2}/\epsilon)s}}\gamma\right)
\right)\nabla \phi  \right\Vert  ^{2}ds \mathbb{E} \int_{0}%
^{\Delta}\left\Vert
u^{\epsilon}_{2}(s)\right\Vert  ^{2} ds }\nonumber\\
&  \rightarrow0, \quad\text{as } \epsilon,\Delta\downarrow 0, \Delta
/\frac{\delta^{2}}{\varepsilon} \uparrow\infty,\nonumber
\end{align}}

by Lemma \ref{L:ErgodicTheorem4}, Condition \ref{A:Assumption1}
and the uniform bound (\ref{Eq:Ubound}). Hence, we obtain that
almost surely with respect to $\gamma\in\Gamma$,
\[
\mathbb{E}\left|  B_{1,2}^{\epsilon,\gamma}\right|  \rightarrow0.
\]

This concludes the proof of (\ref{Eq:ViablePair1}). Next, we treat
(\ref{Eq:ViablePair2}). Consider $\tilde{\phi}\in L^{2}(\Gamma)$ stationary, ergodic
random field on $\mathbb{R}^{d-m}$. Let $\phi(y,\gamma)=\tilde{\phi}(\tau_{y}\gamma)$
and assume that $\phi(\cdot,\gamma)\in C^{2}_{b}(\mathbb{R}^{d-m})$. Define the
formal operators
\[
\mathcal{G}^{0,\gamma}_{x,y,\gamma,z_{1},z_{2}}\phi(y,\gamma)=\left[  g(x,y,\gamma
)+\tau_{1}(y,\gamma) z_{1}+\tau_{2}(y,\gamma) z_{2}\right]  D\phi(y,\gamma)
\]
and
\[
\mathcal{G}^{1,\epsilon,\gamma}_{x,y,\gamma, z_{1},z_{2}}\phi(y,\gamma)=\frac{\epsilon
}{\delta^{2}}\mathcal{L}^{\gamma}\phi(y,\gamma)+\frac{1}{\delta}\mathcal{G}%
^{0,\gamma}_{x,y,z_{1},z_{2}}\phi(y,\gamma)
\]

Following the customary notation we write $\tilde{\mathcal{G}}^{0,\gamma
}_{x,\gamma,z_{1},z_{2}}\tilde{\phi}(\gamma)=\left[  \tilde{g}(x,\gamma)+
\tilde{\tau}_{1}(\gamma) z_{1}+\tilde{\tau}_{2}(\gamma) z_{2}\right]  D\tilde{\phi}(\gamma)$ and analogously for
$\tilde{\mathcal{G}}^{1,\epsilon,\gamma}_{x,\gamma, z_{1},z_{2}}\tilde{\phi}(\gamma)$.

For each fixed $\gamma\in\Gamma$, the process
\begin{align}
M^{\epsilon,\gamma}_{t}  &  =\phi(\bar{Y}^{\epsilon,\gamma}_{t})- \phi(\bar
{Y}^{\epsilon,\gamma}_{0})-\int_{0}^{t}\mathcal{G}^{1,\epsilon,\gamma}%
_{\bar{X}^{\epsilon,\gamma}_{s},\bar{Y}^{\epsilon,\gamma}_{s},u^{\epsilon}%
_{1}(s),u^{\epsilon}_{2}(s)}\phi(\bar{Y}^{\epsilon,\gamma}_{s})ds\nonumber\\
&  =\frac{\sqrt{\epsilon}}{\delta}\int_{0}^{t}\left<  D\phi(\bar
{Y}^{\epsilon,\gamma}_{s}),\tau_{1}(\bar{Y}^{\epsilon,\gamma}_{s})dW_{s}\right>+\frac{\sqrt{\epsilon}}{\delta}\int_{0}^{t}\left<  D\phi(\bar
{Y}^{\epsilon,\gamma}_{s}),\tau_{2}(\bar{Y}^{\epsilon,\gamma}_{s})dB_{s}\right>\nonumber
\end{align}
is an $\mathfrak{F}_{t}-$martingale. Set $h(\epsilon)=\frac{\delta^{2}%
}{\epsilon}$ and write
\begin{align}
h(\epsilon)M^{\epsilon,\gamma}_{t}  &  -h(\epsilon)\left[  \phi(\bar{Y}%
^{\epsilon,\gamma}_{t})- \phi(\bar{Y}^{\epsilon,\gamma}_{0})\right] \nonumber\\
&  +h(\epsilon)\left[  \int_{0}^{t}\frac{1}{\Delta}\left(  \int_{s}^{s+\Delta
}\mathcal{G}^{1,\epsilon,\gamma}_{\bar{X}^{\epsilon,\gamma}_{\theta},\bar
{Y}^{\epsilon,\gamma}_{\theta},u^{\epsilon}_{1}(\theta),u^{\epsilon}_{2}(\theta)}\phi(\bar{Y}^{\epsilon
,\gamma}_{\theta})d\theta\right)  ds -\int_{0}^{t}\mathcal{G}^{1,\epsilon
,\gamma}_{\bar{X}^{\epsilon,\gamma}_{s},\bar{Y}^{\epsilon,\gamma}%
_{s},u^{\epsilon}_{1}(s),u^{\epsilon}_{1}(s)}\phi(\bar{Y}^{\epsilon,\gamma}_{s})ds\right] \nonumber\\
&  =-\frac{\delta}{\epsilon}\int_{0}^{t}\frac{1}{\Delta}\left[  \int
_{s}^{s+\Delta}\left(  \mathcal{G}^{0,\gamma}_{\bar{X}^{\epsilon,\gamma
}_{\theta},\bar{Y}^{\epsilon,\gamma}_{\theta},u^{\epsilon}_{1}(\theta),u^{\epsilon}_{2}(\theta)}\phi(\bar
{Y}^{\epsilon,\gamma}_{\theta}) -\mathcal{G}^{0,\gamma}_{\bar{X}%
^{\epsilon,\gamma}_{s},\bar{Y}^{\epsilon,\gamma}_{\theta},u^{\epsilon}_{1}(\theta),
u^{\epsilon}_{2}(\theta)}\phi(\bar{Y}^{\epsilon,\gamma}_{\theta})\right)  d\theta\right]
ds\nonumber\\
&  \hspace{0.5cm}-\frac{\delta}{\epsilon}\int_{\mathcal{Z}\times\mathcal{Z}\times\Gamma
\times[0,t]}\tilde{\mathcal{G}}^{0,\gamma}_{\bar{X}^{\epsilon,\gamma}%
_{s},\gamma,z_{1},z_{2}}\tilde{\phi}(\gamma)\bar{\mathrm{P}}^{\epsilon,\Delta,\gamma
}(dz_{1}dz_{2}d\gamma ds)\nonumber\\
&  \hspace{0.5cm}-\int_{0}^{t}\frac{1}{\Delta}\int_{s}^{s+\Delta}%
\mathcal{L}^{\gamma}\phi(\bar{Y}^{\epsilon,\gamma}_{\theta},\gamma)d\theta ds
\label{Eq:Part2ofViability}%
\end{align}

The boundedness of $\phi$ and of its derivatives imply that almost surely in
$\gamma\in\Gamma$
\[
\mathbb{E}\left[  \left|  h(\epsilon)M^{\epsilon,\gamma}_{t}\right|
^{2}+\left|  h(\epsilon)\left[  \phi(\bar{Y}^{\epsilon,\gamma}_{t})- \phi(\bar
{Y}^{\epsilon,\gamma}_{0})\right]  \right|  \right]  \rightarrow
0,\quad\text{as }\epsilon\downarrow0
\]

Moreover, we have almost surely in $\gamma\in\Gamma$
\begin{align}
&  h(\epsilon) \mathbb{E}\left|  \int_{0}^{t}\frac{1}{\Delta}\left(
\int_{s}^{s+\Delta}\mathcal{G}^{1,\epsilon,\gamma}_{\bar{X}^{\epsilon,\gamma
}_{\theta},\bar{Y}^{\epsilon,\gamma}_{\theta},u^{\epsilon}_{1}(\theta),u^{\epsilon}_{2}(\theta)}\phi(\bar
{Y}^{\epsilon,\gamma}_{\theta})d\theta\right)  ds -\int_{0}^{t}\mathcal{G}%
^{1,\epsilon,\gamma}_{\bar{X}^{\epsilon,\gamma}_{s},\bar{Y}^{\epsilon,\gamma
}_{s},u^{\epsilon}_{1}(s),u^{\epsilon}_{2}(s)}\phi(\bar{Y}^{\epsilon,\gamma}_{s})ds\right| \nonumber\\
&  \leq h(\epsilon) \mathbb{E}\int_{0}^{\Delta}\left|  \mathcal{G}%
^{1,\epsilon,\gamma}_{\bar{X}^{\epsilon,\gamma}_{s},\bar{Y}^{\epsilon,\gamma
}_{s},u^{\epsilon}_{1}(s),u^{\epsilon}_{2}(s)}\phi(\bar{Y}^{\epsilon,\gamma}_{s})\right|  ds
+h(\epsilon) \mathbb{E}\int_{t}^{t+\Delta}\left|  \mathcal{G}%
^{1,\epsilon,\gamma}_{\bar{X}^{\epsilon,\gamma}_{s},\bar{Y}^{\epsilon,\gamma
}_{s},u^{\epsilon}_{1}(s),u^{\epsilon}_{2}(s)}\phi(\bar{Y}^{\epsilon,\gamma}_{s})\right|  ds\nonumber\\
&  \leq\mathbb{E}\int_{0}^{\Delta}\left|  \mathcal{L}\phi(\bar
{Y}^{\epsilon,\gamma}_{s})\right|  ds+\frac{\delta}{\epsilon}\mathbb{E}%
^{\gamma}\int_{0}^{\Delta}\left|  \mathcal{G}^{0,\gamma}_{\bar{X}%
^{\epsilon,\gamma}_{s},\bar{Y}^{\epsilon,\gamma}_{s},u^{\epsilon}_{1}(s),u^{\epsilon}_{2}(s)}%
\phi(\bar{Y}^{\epsilon,\gamma}_{s})\right|  ds\nonumber\\
&  +\mathbb{E}\int_{t}^{t+\Delta}\left|  \mathcal{L}\phi(\bar
{Y}^{\epsilon,\gamma}_{s})\right|  ds+\frac{\delta}{\epsilon}\mathbb{E}%
^{\gamma}\int_{t}^{t+\Delta}\left|  \mathcal{G}^{0,\gamma}_{\bar{X}%
^{\epsilon,\gamma}_{s},\bar{Y}^{\epsilon,\gamma}_{s},u^{\epsilon}_{1}(s),u^{\epsilon}_{2}(s)}%
\phi(\bar{Y}^{\epsilon,\gamma}_{s})\right|  ds\nonumber\\
&  \leq\Delta C_{0}\left[  1+\frac{\delta}{\epsilon}\mathbb{E}%
\int_{0}^{T}\left\Vert  u^{\epsilon}(s)\right\Vert  ^{2}ds\right] \nonumber\\
&  \rightarrow0, \quad\text{as } \epsilon\downarrow0,\nonumber
\end{align}
due to (\ref{Eq:Ubound}) and $\Delta=\Delta(\epsilon
)\downarrow0$. The constant $C_{0}$ depends on the upper bound of the coefficients and
on $\beta,T$.

The first term on the right hand side of (\ref{Eq:Part2ofViability}) goes to
zero in probability, almost surely with respect to $\gamma\in\Gamma$, due to
continuous dependence of $\mathcal{G}^{0,\gamma}_{x,y,z_{1},z_{2}}\phi(y,\gamma)$ on
$x\in\mathbb{R}^{m}$, tightness of $\bar{X}^{\epsilon,\gamma}$, stationarity
and $\delta/\epsilon\downarrow0$.

The second term on the right hand side of (\ref{Eq:Part2ofViability}) also
goes to zero in probability, almost surely with respect to $\gamma\in\Gamma$,
due to continuous dependence of $\mathcal{G}^{0,\gamma}_{x,y,z_{1},z_{2}}\phi(y,\gamma)$ on
$x\in\mathbb{R}^{m}$, tightness of $(\bar{X}^{\epsilon,\gamma},\mathrm{P}%
^{\epsilon,\Delta,\gamma})$, uniform integrability of $\mathrm{P}%
^{\epsilon,\Delta,\gamma}$ (Lemma \ref{L:TightnessOccupationalMeasures}) and
the fact that $\delta/\epsilon\downarrow0$.

Lastly, we consider the third term on the right hand side of
(\ref{Eq:Part2ofViability}). We have
\begin{align}
\int_{0}^{t}\frac{1}{\Delta}\int_{s}^{s+\Delta}\mathcal{L}^{\gamma}\phi(\bar
{Y}^{\epsilon,\gamma}_{\theta},\gamma)d\theta ds  &  = \int_{0}^{t}\frac
{1}{\Delta}\int_{s}^{s+\Delta}\tilde{\mathcal{L}}\tilde{\phi}(\tau_{\bar
{Y}^{\epsilon,\gamma}_{\theta}}\gamma)d\theta ds\nonumber\\
&  =\int_{0}^{t}\int_{\mathcal{Z}\times\Gamma}\tilde{\mathcal{L}}\tilde
{\phi}(\gamma)\mathrm{P}^{\epsilon,\Delta,\gamma}(dzd\gamma ds)\nonumber
\end{align}

Due to weak convergence of $(\bar{X}^{\epsilon,\gamma},\mathrm{P}%
^{\epsilon,\Delta,\gamma})$, the last term converges, almost surely with
respect to $\phi\in\Gamma$ to $\int_{0}^{t}\int_{\mathcal{Z}\times\mathcal{Z}\times\Gamma
}\tilde{\mathcal{L}}\tilde{\phi}(\gamma)\bar{\mathrm{P}}(dz_{1}dz_{2}d\gamma ds)$. Hence,
since the rest of the terms converge to $0$, as $\epsilon\downarrow0$, we
obtain in probability, almost surely in $\gamma\in\Gamma$
\[
\int_{0}^{t}\int_{\mathcal{Z}\times\mathcal{Z}\times\Gamma}\tilde{\mathcal{L}}\tilde{\phi}%
(\gamma)\bar{\mathrm{P}}(dz_{1}dz_{2}d\gamma ds) =0
\]
for almost all $t\in[0,T]$, which together with continuity in $t\in[0,T]$
conclude the proof of (\ref{Eq:ViablePair2}).

\section{Compactness of level sets and quenched lower and upper bounds}
\label{S:LDPrelaxedForm}

Compactness of level sets of the rate function is standard and will not be
repeated here (e.g., Subsection 4.2. of \cite{DupuisSpiliopoulos} or
\cite{FW1}).

Let us now prove the quenched lower bound. First we remark that we can
restrict attention to controls that satisfy Conditions \ref{Eq:UniformlySquareIntegrableControlsAdditional} and \ref{Eq:UniformlySquareIntegrableControlsAdditionalb},
which  are required in order for  Lemma \ref{L:ErgodicTheorem4} to be true. For this purpose we have the following lemma, whose proof is deferred to the end of this section.

\begin{lemma}\label{L:RestrictingTheControl}
Let $(\bar{X}_{s}^{\epsilon,\gamma},
\bar{Y}_{s}^{\epsilon,\gamma})$ be the strong solution to (\ref{Eq:Main2}) and assume Conditions \ref{A:Assumption1} and \ref{A:Assumption2}. Then, the infimum of the representation in (\ref{Eq:VariationalRepresentation}) can be taken over all controls that satisfy Conditions
\ref{Eq:UniformlySquareIntegrableControlsAdditional} and \ref{Eq:UniformlySquareIntegrableControlsAdditionalb}.
\end{lemma}

Based on Lemma \ref{L:RestrictingTheControl}, we can restrict attention to controls satisfying Conditions
\ref{Eq:UniformlySquareIntegrableControlsAdditional} and \ref{Eq:UniformlySquareIntegrableControlsAdditionalb}. Given such controls, we construct the controlled pair $(\bar
{X}^{\epsilon,\gamma},\mathrm{P}^{\epsilon,\Delta,\gamma})$ based on such a
family of controls. Then, Theorem \ref{T:MainTheorem1} implies tightness of
the pair $\left\{  (\bar{X}^{\epsilon,\gamma},\mathrm{P}^{\epsilon
,\Delta,\gamma}), \epsilon,\Delta>0 \right\}  $. Let us denote by $(\bar
{X},\bar{\mathrm{P}})\in\mathcal{V}$ an accumulation point of the controlled
pair in distribution, almost surely with respect to $\gamma\in\Gamma$. Then,
by Fatou's lemma we conclude the proof of the lower bound. Indeed
\begin{align}
\liminf_{\epsilon\downarrow0}-\epsilon\log\mathbb{E}\left[
e^{-\frac{1}{\epsilon}h(X^{\epsilon})}\right]   &  \geq\liminf_{\epsilon
\downarrow0}\mathbb{E}\left[  \frac{1}{2}\int_{0}%
^{T}\left[\left\Vert u^{\epsilon,\gamma}_{1}(s)\right\Vert^{2}+\left\Vert u^{\epsilon,\gamma}_{2}(s)\right\Vert^{2}\right]ds+h(\bar{X}^{\epsilon,\gamma})\right]
\nonumber\\
&  \geq\liminf_{\epsilon\downarrow0}\mathbb{E}\left[
\frac{1}{2}\int_{0}^{T}\int_{\mathcal{Z}\times\mathcal{Z}\times\Gamma}\left[\left\Vert z_{1}\right\Vert^{2}+\left\Vert z_{2}\right\Vert^{2}\right]\mathrm{P}%
^{\epsilon,\Delta,\gamma}(dz_{1}dz_{2} d\gamma ds)+h(\bar{X}^{\epsilon,\gamma})\right]
\nonumber\\
&  \geq\inf_{(\phi,\mathrm{P})\in\mathcal{V}  }\left[
\frac{1}{2}\int_{\mathcal{Z}\times\mathcal{Z}\times\Gamma\times[0,T]}\left[\left\Vert z_{1}\right\Vert^{2}+\left\Vert z_{2}\right\Vert^{2}\right]\mathrm{P}(dz_{1}dz_{2}
d\gamma ds)+h(\phi)\right] \nonumber
\end{align}
which concludes the proof of the Laplace principle lower bound.

It remains to prove the quenched upper bound for the Laplace principle. To do
so, we fix a bounded and continuous function $h:\mathcal{C}\left(
[0,T];\mathbb{R}^{m}\right)  \mapsto\mathbb{R}$, and we show that
\[
\limsup_{\epsilon\downarrow0}-\epsilon\log\mathbb{E}\left[
e^{-\frac{1}{\epsilon}h(X^{\epsilon})}\right]  \leq\inf_{\phi\in
\mathcal{C}\left(  [0,T];\mathbb{R}^{m}\right)  }\left\{  S(\phi
)+h(\phi)\right\}
\]

The idea is to fix a nearly optimizer of the right hand side of the last
display and construct the control which attains the given upper bound. Fix
$\eta>0$ and consider $\psi\in\mathcal{C}\left(  [0,T];\mathbb{R}^{m}\right)
$ with $\psi_{0}=x_{0}$ such that
\[
S(\psi)+h(\psi)\leq\inf_{\phi\in\mathcal{C}\left(  [0,T];\mathbb{R}%
^{m}\right)  }\left\{  S(\phi)+h(\phi)\right\}  +\eta<\infty
\]

Boundedness of $h$ implies that $S(\psi)<\infty$ which means that $\psi$ is
absolutely continuous. Since the local rate function $L^{o}(x,v)$, defined in
(\ref{Eq:LocalRateFunction}), is continuous and bounded as a function of
$(x,v)\in\mathbb{R}^{m}$, standard mollification arguments (Lemmas 6.5.3 and
6.5.5 in \cite{DupuisEllis}) allow to assume that $\dot{\psi}$ is piecewise
constant. Next, we define the elements of $L^{2}(\Gamma)$%

\[
\tilde{u}_{1,\rho}(t,x,\gamma)=\left(\tilde{\sigma}(x,\gamma)+ D\tilde
{\chi}_{\rho}(\gamma)\tilde{\tau}_{1}(\gamma)\right)^{T}q^{-1}(x)(\dot{\psi}_{t}-r(x))
\]
and
\[
\tilde{u}_{2,\rho}(t,x,\gamma)=\left(D\tilde
{\chi}_{\rho}(\gamma)\tilde{\tau}_{2}(\gamma)\right)^{T}q^{-1}(x)(\dot{\psi}_{t}-r(x))
\]

and the associated stationary fields $u_{1,\rho}(t,x,y,\gamma)=\tilde{u}_{1,\rho
}(t,x,\tau_{y}\gamma)$ and $u_{2,\rho}(t,x,y,\gamma)=\tilde{u}_{2,\rho
}(t,x,\tau_{y}\gamma)$. We recall that $\tilde{\chi}_{\rho}$ satisfies the
auxiliary problem in (\ref{Eq:RandomCellProblem}). Let us consider now the
solution $\left(\bar{X}^{\epsilon}_{t},\bar{Y}^{\epsilon}_{t}\right)$ of (\ref{Eq:Main2}) with the control
$u(t)=\left(u_{1}(t),u_{2}(t)\right)$ being
\[
u^{\epsilon,\rho,\gamma}_{t}=\left(u_{1,\rho}\left(  t,\bar{X}^{\epsilon}_{t}%
,\bar{Y}^{\epsilon}_{t},\gamma\right), u_{2,\rho}\left(  t,\bar{X}^{\epsilon}_{t}%
,\bar{Y}^{\epsilon}_{t},\gamma\right)\right).
\]

Then, replacing $c(x,y,\gamma)$ by $c(t,x,y,\gamma)=c(x,y,\gamma)+\sigma(y,\gamma)u_{1,\rho}(t,x,y,\gamma)$, and $g(x,y,\gamma)$ by $g(t,x,y,\gamma)=g(x,y,\gamma)+\tau_{1}(y,\gamma)u_{1,\rho}(t,x,y,\gamma)+\tau_{2}(y,\gamma)u_{2,\rho}(t,x,y,\gamma)$ Theorem \ref{L:ErgodicTheorem4}
implies that
\begin{equation*}
\bar{X}^{\epsilon}\rightarrow\bar{X} \quad\text{in law, almost surely
with respect to }\gamma\in\Gamma,
\end{equation*}
as $\epsilon\downarrow0$ where we have that w.p. 1 the limit is
\begin{align}
\bar{X}_{t}  &  =x_{0}+\int_{0}^{t}\lim_{\rho\downarrow0}\mathrm{E}^{\pi
}\left[ \tilde{c}(\bar{X}_{s},\cdot)+D\tilde{\chi}_{\rho}(\cdot)\tilde{g}(\bar{X}_{s},\cdot) + \left(\tilde{\sigma}(\bar{X}_{s},\cdot)+D\tilde{\chi}_{\rho}(\cdot)\tilde{\tau}_{1}(\cdot)\right) \tilde{u}_{1,\rho}(s,\bar{X}_{s},\cdot)\right.\nonumber\\
&\qquad\qquad\left.+\left(D\tilde{\chi}_{\rho}(\cdot)\tilde{\tau}_{2}(\cdot)\right) \tilde{u}_{2,\rho}(s,\bar{X}_{s},\cdot)\right]
ds\nonumber\\
&  =x_{0}+\int_{0}^{t}\lim_{\rho\downarrow0}\mathrm{E}^{\pi}\left[
\tilde{c}(\bar{X}_{s},\cdot)+D\tilde{\chi}_{\rho}(\cdot)\tilde{g}(\bar{X}_{s},\cdot)\right]
ds\nonumber\\
&\qquad\qquad+\int_{0}^{t}\lim_{\rho\downarrow0}\mathrm{E}^{\pi}\left[  \left(\tilde{\sigma}(\bar{X}_{s},\cdot)+D\tilde{\chi}_{\rho}(\cdot)\tilde{\tau}_{1}(\cdot)\right) \tilde{u}_{1,\rho}(s,\bar{X}_{s},\cdot)+\left(D\tilde{\chi}_{\rho}(\cdot)\tilde{\tau}_{2}(\cdot)\right) \tilde{u}_{2,\rho}(s,\bar{X}_{s},\cdot)\right]  ds\nonumber\\
&  =x_{0}+\int_{0}^{t}r(\bar{X}_{s})ds+\int_{0}^{t}\lim_{\rho\downarrow
0}\mathrm{E}^{\pi}\left\{ \left[ \left(\tilde{\sigma}(\bar{X}_{s},\cdot)
+D\tilde{\chi}_{\rho}(\cdot)\tilde{\tau}_{1}(\cdot)\right) \left(\tilde{\sigma}(\bar{X}_{s},\cdot)+D\tilde{\chi}_{\rho}(\cdot)\tilde{\tau}_{1}(\cdot)\right)^{T} \right.\right.\nonumber\\
 &\hspace{6cm}\left.\left.+\left(D\tilde{\chi}_{\rho}(\cdot)\tilde{\tau}_{2}(\cdot)\right) \left(D\tilde{\chi}_{\rho}(\cdot)\tilde{\tau}_{2}(\cdot)\right)^{T}\right]q^{-1}(\bar{X}_{s})
(\dot{\psi}_{s}-r(\bar{X}_{s}))\right\}  ds\nonumber\\
&  =x_{0}+\int_{0}^{t}r(\bar{X}_{s})ds+\int_{0}^{t}\mathrm{E}^{\pi}\left\{ \left[ \left(\tilde{\sigma}(\bar{X}_{s},\cdot)+D\tilde{\xi}(\cdot)\tilde{\tau}_{1}(\cdot)\right) \left(\tilde{\sigma}(\bar{X}_{s},\cdot)+D\tilde{\xi}(\cdot)\tilde{\tau}_{1}(\cdot)\right)^{T} \right.\right.\nonumber\\ &\hspace{6cm}\left.\left.+\left(D\tilde{\xi}(\cdot)\tilde{\tau}_{2}(\cdot)\right) \left(D\tilde{\xi}(\cdot)\tilde{\tau}_{2}(\cdot)\right)^{T}\right]q^{-1}(\bar{X}_{s})
(\dot{\psi}_{s}-r(\bar{X}_{s}))\right\}ds
\nonumber\\
&  =x_{0}+\int_{0}^{t}r(\bar{X}_{s})ds+\int_{0}^{t}q(\bar{X}_{s}) q^{-1}(\bar{X}_{s})(\dot{\psi}%
_{s}-r(\bar{X}_{s}))ds\nonumber\\
&  =x_{0}+\psi_{t}-\psi_{0}\nonumber\\
&  =\psi_{t}.\nonumber
\end{align}

Moreover, by Theorem \ref{L:ErgodicTheorem4}  we
have that for any $\eta>0$, there exists a $N_{\eta}$ with $\nu\left[
N_{\eta}\right]  >1-\eta$ such that%

\begin{align}
&\lim_{\epsilon\downarrow0}\sup_{\gamma\in N_{\eta}}\sup_{0\leq t\leq
T}\mathbb{E}\left[  \frac{1}{2}\int_{0}^{T}\left[\left\Vert u^{\epsilon
,\gamma}_{1,\rho}(s)\right\Vert^{2}+\left\Vert u^{\epsilon
,\gamma}_{2,\rho}(s)\right\Vert^{2}\right]ds\right.\nonumber\\
&\hspace{4cm}\left.- \frac{1}{2}\int_{0}^{T}\lim_{\rho\downarrow
0}\mathrm{E}^{\pi}\left[\left\Vert u^{\epsilon
,\gamma}_{1,\rho}(s,X^{\epsilon}_{s},\cdot)\right\Vert^{2}+\left\Vert u^{\epsilon
,\gamma}_{2,\rho}(s,X^{\epsilon}_{s},\cdot)\right\Vert^{2}\right]
ds\right|  =0\nonumber
\end{align}

Therefore, noticing that for each fixed $x\in\mathbb{R}^{m}$ and almost every
$t\in[0,T]$
\begin{align}
  \lim_{\rho\downarrow0}\mathrm{E}^{\pi}\left[\left\Vert u^{\epsilon
,\gamma}_{1,\rho}(s,x,\cdot)\right\Vert^{2}+\left\Vert u^{\epsilon
,\gamma}_{2,\rho}(s,x,\cdot)\right\Vert^{2}\right]
ds&=
(\dot{\psi}_{s}-r(x))^{T}q^{-1}(x)q(x) q^{-1}(x)(\dot{\psi}%
_{s}-r(x))\nonumber\\
& =L^{0}(x,\dot{\psi}_{s}),\nonumber
\end{align}
we finally obtain%

\begin{align}
\limsup_{\epsilon\downarrow0}-\epsilon\log\mathbb{E}\left[
e^{-\frac{1}{\epsilon}h(X^{\epsilon})}\right]   &  =\limsup_{\epsilon
\downarrow0}\inf_{u\in\mathcal{A}}\mathbb{E}\left[  \frac
{1}{2}\int_{0}^{T}\left[\left\Vert u_{1}(s)\right\Vert^{2}+\left\Vert u_{2}(s)\right\Vert^{2}\right]ds+h(\bar{X}^{\epsilon})\right] \nonumber\\
&  \leq\limsup_{\epsilon\downarrow0}\mathbb{E}\left[
\frac{1}{2}\int_{0}^{T}\left[\left\Vert u^{\epsilon,\gamma}_{1,\rho}(s)\right\Vert^{2}+\left\Vert u^{\epsilon,\gamma}_{2,\rho}(s)\right\Vert^{2}\right]ds+h(\bar{X}^{\epsilon
})\right] \nonumber\\
&  \leq\left[  S(\psi)+h(\psi)\right] \nonumber\\
&  \leq\inf_{\phi\in\mathcal{C}\left(  [0,T];\mathbb{R}^{m}\right)  }\left\{
S(\phi)+h(\phi)\right\}  +\eta.\nonumber
\end{align}

The first line follows from the representation (\ref{Eq:VariationalRepresentation}) and the second line from the choice of the particular control. The third
line follows from he convergence of the $X^{\epsilon}$ and of the cost
functional using the continuity of $h$. Then, the fourth line follows from the
fact $\bar{X}_{t}=\psi_{t}$. Since the last statement is true for every
$\eta>0$ the proof of the upper bound is done.

We conclude this section with the proof of Lemma \ref{L:RestrictingTheControl}.

\begin{proof}[Proof of Lemma \ref{L:RestrictingTheControl}]
First, we explain why Condition
\ref{Eq:UniformlySquareIntegrableControlsAdditional} can be assumed without loss of generality. Without loss of generality, we can consider a function $h(x)$ that is bounded and uniformly Lipschitz
continuous in $\mathbb{R}^{m}$. Namely, there exists a constant $L_{h}$ such that
\[
|h(x)-h(y)|\leq L_{h}\left\Vert x-y\right\Vert
\]
and $\left\Vert h\right\Vert_{\infty}=\sup_{x\in\mathbb{R}^{m}}|h(x)|<\infty$. We recall that the representation
\begin{equation}
-\epsilon\log\mathbb{E}\left[  e^{-\frac{1}{\epsilon
}h(X^{\epsilon}_{T})}\right]  =\inf_{u\in\mathcal{A}}\mathbb{E}\left[  \frac{1}{2}\int_{0}^{T}\left\Vert u(s)\right\Vert^{2}ds+h(\bar{X}^{\epsilon
}_{T})\right]   \label{Eq:RepresentationTheorem1}%
\end{equation}
is valid in a $\gamma$ by $\gamma$ basis.

Fix $a>0$. Then for every $\epsilon>0$, there exists a control $u^{\epsilon}\in\mathcal{A}$ such that
\begin{equation}
-\epsilon\log\mathbb{E}\left[  e^{-\frac{1}{\epsilon
}h(X^{\epsilon}_{T})}\right]  \geq \mathbb{E}\left[  \frac{1}{2}\int_{0}^{T}\left\Vert u^{\epsilon}(s)\right\Vert^{2}ds+h(\bar{X}^{\epsilon
}_{T})\right]-a  . \label{Eq:RepresentationTheorem1a}%
\end{equation}

So, letting $M_{0}=\left\Vert h\right\Vert_{\infty}=\sup_{x\in\mathbb{R}^{m}}|h(x)|$ we easily see that such a control $u^{\epsilon}$ should satisfy
\[
\sup_{\epsilon>0,\gamma\in\Gamma}\mathbb{E}\left[  \frac{1}{2}\int_{0}^{T}\left\Vert u^{\epsilon}(s)\right\Vert^{2} ds\right]\leq M_{1}=2M_{0}+a.
\]

Given that the latter bound has been established, the claim that in proving the Laplace principle lower bound one can assume Condition
\ref{Eq:UniformlySquareIntegrableControlsAdditional} without loss of
generality, follows by the last display and the
representation (\ref{Eq:RepresentationTheorem1}) as in the proof of Theorem
4.4 of \cite{BudhirajaDupuis2000}. In particular, it follows by the arguments in \cite{BudhirajaDupuis2000} that if
the last display holds, then it is enough to assume that for given $a>0$
the controls satisfy the bound
\[
\int_{0}^{T}\left\Vert u^{\epsilon}(s)\right\Vert^{2}ds<N,
\]
with
\[
N\geq \frac{4M_{0}(4M_{0}+a)}{a}
\]
which proves that in proving the Laplace principle lower bound one can assume Condition
\ref{Eq:UniformlySquareIntegrableControlsAdditional} without loss of
generality.

Second, we explain why Condition
\ref{Eq:UniformlySquareIntegrableControlsAdditionalb} can be assumed without loss of generality. It is clear by the representation (\ref{Eq:RepresentationTheorem1}) that the
trivial bound holds
\begin{equation*}
-\epsilon\log\mathbb{E}\left[  e^{-\frac{1}{\epsilon
}h(X^{\epsilon}_{T})}\right]  \leq\mathbb{E} h(X^{\epsilon
}_{T}),
\end{equation*}
where the control $u^{\epsilon}(\cdot)=0$ is used to evaluate the right hand
side.  Thus, we only need to consider controls that
satisfy
\begin{align}
\mathbb{E}\left[  \frac{1}{2}\int_{0}^{T} \left\Vert u^{\epsilon,\gamma}(s)\right\Vert^{2}ds+h(\bar{X}^{\epsilon}_{T})\right]   &  \leq \mathbb{E} h(X^{\epsilon
}_{T})\nonumber
\end{align}
which by the Lipschitz assumption on $h$, implies that
\begin{align}
\mathbb{E}\left[  \int_{0}^{T}\left\Vert u^{\epsilon,\gamma}(s)\right\Vert^{2}ds\right]   &  \leq \mathbb{E} \left|h(X^{\epsilon}_{T})-h(\bar{X}^{\epsilon}_{T})\right|\nonumber\\
&  \leq L_{h} \mathbb{E} \left\Vert X^{\epsilon}_{T}-\bar{X}^{\epsilon}_{T}\right\Vert.\nonumber
\end{align}

Let us next define the processes $\left(\hat{\bar{X}}^{\epsilon}_{t},\hat{\bar{Y}}^{\epsilon}_{t}\right)=\left(\bar{X}^{\epsilon}_{\frac{\delta^{2}t}{\epsilon}},\bar{Y}^{\epsilon}_{\frac{\delta^{2}t}{\epsilon}}\right)$  and
$\left(\hat{X}^{\epsilon}_{t},\hat{Y}^{\epsilon}_{t}\right)=\left(X^{\epsilon}_{\frac{\delta^{2}t}{\epsilon}},Y^{\epsilon}_{\frac{\delta^{2}t}{\epsilon}}\right)$. It is easy to see that $\left(\hat{\bar{X}}^{\epsilon}_{t},\hat{\bar{Y}}^{\epsilon}_{t}\right)$ satisfies the SDE

\begin{eqnarray}
d\hat{\bar{X}}^{\epsilon}_{t}&=&  \delta b\left(\hat{\bar{Y}}^{\epsilon}_{t},\gamma\right)dt+\frac{\delta^{2}}{\epsilon}\left[c\left(  \hat{\bar{X}}^{\epsilon}_{t}%
,\hat{\bar{Y}}^{\epsilon}_{t},\gamma\right)+\sigma\left(  \hat{\bar{X}}_{t}^{\epsilon},\hat{\bar{Y}}_{t}^{\epsilon},\gamma\right)  u_{1}(\delta^{2}t/\epsilon)\right]   dt+\delta%
\sigma\left(  \hat{\bar{X}}^{\epsilon}_{t},\hat{\bar{Y}}^{\epsilon}_{t},\gamma\right)
dW_{t}, \nonumber\\
d\hat{\bar{Y}}^{\epsilon}_{t}&=& f\left(\hat{\bar{Y}}^{\epsilon}_{t},\gamma\right)dt  +\frac{\delta}{\epsilon}\left[ g\left(  \hat{\bar{X}}^{\epsilon}_{t}
,\hat{\bar{Y}}^{\epsilon}_{t},\gamma\right)+\tau_{1}\left(\hat{\bar{Y}}^{\epsilon}_{t},\gamma\right)u_{1}(\delta^{2}t/\epsilon)+\tau_{2}\left(\hat{\bar{Y}}^{\epsilon}_{t},\gamma\right)u_{2}(\delta^{2}t/\epsilon)\right]   dt\nonumber\\
& &\hspace{5cm}+\left[
\tau_{1}\left(\hat{\bar{Y}}^{\epsilon}_{t},\gamma\right)
dW_{t}+\tau_{2}\left(\hat{\bar{Y}}^{\epsilon}_{t},\gamma\right)dB_{t}\right],\nonumber\\
\hat{\bar{X}}^{\epsilon}_{0}&=&x_{0},\hspace{0.2cm}\hat{\bar{Y}}^{\epsilon}_{0}=y_{0},\nonumber
\end{eqnarray}
and $\left(\hat{X}^{\epsilon}_{t},\hat{Y}^{\epsilon}_{t}\right)$ satisfies the same SDE with the control $u^{\epsilon}_{1}(\cdot)=u^{\epsilon}_{2}(\cdot)=0$.

So, we basically have that
\begin{align}
\frac{1}{\epsilon}\mathbb{E}\left[  \int_{0}^{\frac{\delta^{2}T}{\epsilon}}\left\Vert u^{\epsilon,\gamma}(s)\right\Vert^{2}ds\right]   &  \leq L_{h} \frac{1}{\epsilon} \mathbb{E} \left\Vert X^{\epsilon}_{\frac{\delta^{2}T}{\epsilon}}-\bar{X}^{\epsilon}_{\frac{\delta^{2}T}{\epsilon}}\right\Vert\nonumber\\
&= L_{h} \frac{\delta}{\epsilon} \mathbb{E} \left\Vert \frac{1}{\delta} \hat{X}^{\epsilon}_{T}-\frac{1}{\delta}\hat{\bar{X}}^{\epsilon}_{T}\right\Vert.\nonumber
\end{align}

For notational convenience, we define
\[
\nu^{\epsilon}_{T}\doteq \frac{1}{\epsilon}\mathbb{E}\left[  \int_{0}^{\frac{\delta^{2}T}{\epsilon}}\left\Vert u^{\epsilon,\gamma}(s)\right\Vert^{2}ds\right]
\]
and
\[
m^{\epsilon}_{T}\doteq \mathbb{E} \left\Vert \frac{1}{\delta}\hat{X}^{\epsilon}_{T}-\frac{1}{\delta}\hat{\bar{X}}^{\epsilon}_{T}\right\Vert^{2}+\mathbb{E} \left\Vert \hat{Y}^{\epsilon}_{T}-\hat{\bar{Y}}^{\epsilon}_{T}\right\Vert^{2}.
\]

Since for $x>0$, the function $x^{2}$ is increasing, the latter inequality, followed by Jensen's inequality give us
\begin{align}
\left|\nu^{\epsilon}_{T}\right|^{2}\leq \left|L_{h} \frac{\delta}{\epsilon}\right|^{2}m^{\epsilon}_{T}.\nonumber
\end{align}

The next step is to derive an upper bound of $m^{\epsilon}_{T}$ in terms of $|\nu^{\epsilon}_{T}|^{2}$. Writing down the differences of
 $\hat{X}^{\epsilon}_{T}-\hat{\bar{X}}^{\epsilon}_{T}$ and $\hat{Y}^{\epsilon}_{T}-\hat{\bar{Y}}^{\epsilon}_{T}$, squaring, taking expectation and using Lipschitz continuity of the functions $b,c,f,g,\sigma,\tau_{1},\tau_{2}$ and boundedness of $\sigma,\tau_{1},\tau_{2}$ we obtain the inequality
\begin{align}
m^{\epsilon}_{T}&\leq C_{0}\int_{0}^{T} m^{\epsilon}_{s}ds+
 C_{1}\left\{\left|\frac{\delta}{\epsilon}\mathbb{E} \left[\int_{0}^{T}\left\Vert u^{\epsilon,\gamma}_{1}\left(\frac{\delta^{2}s}{\epsilon}\right)\right\Vert ds\right]\right|^{2}
 +\left|\frac{\delta}{\epsilon}\mathbb{E} \left[\int_{0}^{T}\left\Vert u^{\epsilon,\gamma}_{2}\left(\frac{\delta^{2}s}{\epsilon}\right)\right\Vert ds\right]\right|^{2}\right\},\nonumber
\end{align}
where the constants $C_{0}, C_{1}$ depends only on the Lipschitz constants of $b,c,f,g,\sigma,\tau_{1},\tau_{2}$ and on the sup norm of $\sigma,\tau_{1},\tau_{2}$. Defining for notational convenience
\[
\left|a^{\epsilon}_{T}\right|^{2}\doteq \frac{\delta}{\epsilon}\mathbb{E}\left[  \int_{0}^{T}\left\Vert u^{\epsilon,\gamma}_{1}\left(\frac{\delta^{2}s}{\epsilon}\right)\right\Vert ds\right]+\frac{\delta}{\epsilon}\mathbb{E} \left[\int_{0}^{T}\left\Vert u^{\epsilon,\gamma}_{2}\left(\frac{\delta^{2}s}{\epsilon}\right)\right\Vert ds\right].
\]

Gronwall lemma, gives us
\begin{align}
m^{\epsilon}_{T}&\leq C_{1}\left|a^{\epsilon}_{T}\right|^{2}+ C_{0}C_{1}\int_{0}^{T}\left|a^{\epsilon}_{s}\right|^{2}e^{C_{0}(T-s)}ds.\nonumber
\end{align}

Let us now rewrite and upper bound $\left|a^{\epsilon}_{T}\right|^{2}$. We notice that, H\"{o}lder inequality followed by Young's inequality give us
\begin{align}
\left|a^{\epsilon}_{T}\right|^{2}&=\left|\frac{\epsilon}{\delta}\frac{1}{\epsilon}\mathbb{E}\left[  \int_{0}^{\frac{\delta^{2}T}{\epsilon}}\left\Vert u^{\epsilon,\gamma}_{1}\left(s\right)\right\Vert ds\right]\right|^{2}+\left|\frac{\epsilon}{\delta}\frac{1}{\epsilon}\mathbb{E}\left[  \int_{0}^{\frac{\delta^{2}T}{\epsilon}}\left\Vert u^{\epsilon,\gamma}_{2}\left(s\right)\right\Vert ds\right]\right|^{2}\nonumber\\
&\leq \frac{1}{\delta^{2}}\frac{\delta^{2}T}{\epsilon}\mathbb{E}\left[  \int_{0}^{\frac{\delta^{2}T}{\epsilon}}\left\Vert u^{\epsilon,\gamma}\left(s\right)\right\Vert^{2}ds\right]\nonumber\\
&= T \frac{1}{\epsilon}\mathbb{E}\left[  \int_{0}^{\frac{\delta^{2}T}{\epsilon}}\left\Vert u^{\epsilon,\gamma}\left(s\right)\right\Vert^{2}ds\right]\nonumber\\
&= T\nu^{\epsilon}_{T}\nonumber\\
&\leq \frac{T^{2}}{2}+\frac{\left|\nu^{\epsilon}_{T}\right|^{2}}{2}.\nonumber
\end{align}

Putting these estimates together, we obtain
\begin{align}
\left|\nu^{\epsilon}_{T}\right|^{2}&  \leq L_{h}^{2} \left|\frac{\delta}{\epsilon}\right|^{2} m^{\epsilon}_{T}\nonumber\\
&\leq L_{h}^{2}C_{1} \left|\frac{\delta}{\epsilon}\right|^{2} \left[ \left|a^{\epsilon}_{T}\right|^{2}+ C_{0}\int_{0}^{T}\left|a^{\epsilon}_{s}\right|^{2}e^{C_{0}(T-s)}ds\right]\nonumber\\
&\leq  L_{h}^{2}C_{1} \left|\frac{\delta}{\epsilon}\right|^{2} \left[ \left(\frac{T^{2}}{2}+\frac{\left|\nu^{\epsilon}_{T}\right|^{2}}{2}\right)+ C_{0}\int_{0}^{T}\left(\frac{s^{2}}{2}+\frac{\left|\nu^{\epsilon}_{s}\right|^{2}}{2}\right)e^{C_{0}(T-s)}ds\right].\nonumber
\end{align}

Therefore, by choosing $\delta/\epsilon$ sufficiently small such that $L_{h}^{2}C_{1} \left|\frac{\delta}{\epsilon}\right|^{2}\leq 1$, we have
\begin{align*}
\frac{\left|\nu^{\epsilon}_{T}\right|^{2}}{2}
&\leq L_{h}^{2}C_{1} \left|\frac{\delta}{\epsilon}\right|^{2} \left[ \frac{T^{2}}{2} + C_{0}\int_{0}^{T}\left(\frac{s^{2}}{2}+\frac{\left|\nu^{\epsilon}_{s}\right|^{2}}{2}\right)e^{C_{0}(T-s)}ds\right]\nonumber\\
&\leq L_{h}^{2}C_{1} \left|\frac{\delta}{\epsilon}\right|^{2} \left[ \frac{T^{2}}{2} + \frac{T^{2}}{2}(e^{C_{0}T}-1) +C_{0} \int_{0}^{T}\frac{\left|\nu^{\epsilon}_{s}\right|^{2}}{2}e^{C_{0}(T-s)}ds\right]\nonumber\\
&= L_{h}^{2}C_{1} \left|\frac{\delta}{\epsilon}\right|^{2} \left[   \frac{T^{2}}{2}e^{C_{0}T} +C_{0} \int_{0}^{T}\frac{\left|\nu^{\epsilon}_{s}\right|^{2}}{2}e^{C_{0}(T-s)}ds\right].
\end{align*}

Thus, we have
\begin{align*}
e^{-C_{0}T}\frac{\left|\nu^{\epsilon}_{T}\right|^{2}}{2}
&\leq L_{h}^{2}C_{1} \left|\frac{\delta}{\epsilon}\right|^{2} \left[   \frac{T^{2}}{2} +C_{0} \int_{0}^{T}e^{-C_{0}s}\frac{\left|\nu^{\epsilon}_{s}\right|^{2}}{2}ds\right].
\end{align*}

So, letting $\beta^{\epsilon}_{T}=L_{h}^{2}C_{1} \frac{T^{2}}{2}\left|\frac{\delta}{\epsilon}\right|^{2} $ and $\theta^{\epsilon}=L_{h}^{2}C_{1} C_{0}\left|\frac{\delta}{\epsilon}\right|^{2}$, Gronwall lemma guarantees that
\begin{align*}
e^{-C_{0}T}\frac{\left|\nu^{\epsilon}_{T}\right|^{2}}{2}&\leq \beta^{\epsilon}_{T}+\theta^{\epsilon} \int_{0}^{T}\beta^{\epsilon}_{s}e^{\theta^{\epsilon}(T-s)}ds.
\end{align*}

Since $\beta^{\epsilon}_{T}$ and $\theta^{\epsilon}$ go uniformly in $\gamma\in\Gamma$ to zero at the speed $O(\left(\frac{\delta}{\epsilon}\right)^{2})$ as $\epsilon\downarrow 0$, we get that
\[
\left|\nu^{\epsilon}_{T}\right|^{2}\leq C (\delta/\epsilon)^{2},
\]
where the constant $C$, depends on $T$, but not on $\epsilon,\delta$ or $\gamma$. This concludes the argument of why Condition
\ref{Eq:UniformlySquareIntegrableControlsAdditionalb} can be assumed without loss of generality.
\end{proof}

\section{Proof of Theorem \ref{T:MainTheorem3} }\label{S:ProofTheoremExplicitLDP}

In this section we prove that the explicit expression of the large deviation's action functional is given by Theorem \ref{T:MainTheorem3}.

Due to Theorem \ref{T:MainTheorem2},
we only need to prove that the rate function given in
(\ref{Eq:GeneralRateFunction}) can be written in the form
of Theorem \ref{T:MainTheorem3}. First, we notice that one can write
(\ref{Eq:GeneralRateFunction}) in terms of a local rate function,
in the form
\[
S(\phi)=\int_{0}^{T}L^{r}(\phi_{s},\dot{\phi}_{s})ds
\]
where we have defined
\[
L^{r}(x,v)=\inf_{\mathrm{P}\in\mathcal{A}_{x,v}^{r}}\frac{1}{2}\int
_{\mathcal{Z}\times\mathcal{Z}\times\Gamma}\left[\left\Vert z_{1}\right\Vert^{2}+\left\Vert z_{2}\right\Vert^{2}\right]\mathrm{P}(dz_{1}dz_{2}d\gamma)
\]
and
\begin{align*}
\mathcal{A}_{x,v}^{r}  &  =\left\{  \mathrm{P}\in\mathcal{P}\left(
\mathcal{Z}\times\mathcal{Z}\times\Gamma\right)  :\int_{\mathcal{Z}\times\mathcal{Z}\times\Gamma}\tilde{L}%
\tilde{f}(\gamma)\mathrm{P}(dz_{1}dz_{2}d\gamma)=0,\quad\forall\quad\tilde{f}%
\in\mathcal{D}(\tilde{L})\right. \\
& \hspace{0.2cm} \left.  \int_{\mathcal{Z}\times\mathcal{Z}\times\Gamma}\left[\left\Vert z_{1}\right\Vert^{2}+\left\Vert z_{2}\right\Vert^{2}\right]\mathrm{P}%
(dz_{1}dz_{2}d\gamma)<\infty,\text{ and }v=\lim_{\rho\rightarrow0}\int_{\mathcal{Z}%
\times\mathcal{Z}\times\Gamma}\tilde{\lambda}_{\rho}(x,\gamma,z_{1},z_{2})\mathrm{P}(dz_{1}dz_{2}d\gamma)\right\}
\end{align*}

This follows directly by the definition of a viable pair (Definition
\ref{D:ViablePair}). We call this representation the \textquotedblleft
relaxed\textquotedblright\ formulation since the control is characterized as a
distribution on $\mathcal{Z}\times\mathcal{Z}$ rather than an element of $\mathcal{Z}\times\mathcal{Z}$.
However, as we shall demonstrate below, the structure of the problem allows us
to rewrite the relaxed formulation of the local rate function in terms of an
ordinary formulation of an equivalent local rate function, where the control
is indeed given as an element of $\mathcal{Z}\times\mathcal{Z}$. In preparation for this
representation, we notice that any element $\mathrm{P}\in\mathcal{P}\left(
\mathcal{Z}\times\mathcal{Z}\times \Gamma\right)  $ can be written of a stochastic kernel on
$\mathcal{Z}\times\mathcal{Z}$ given $\Gamma$ and a probability measure on $\Gamma$, namely
\[
\mathrm{P}(dz_{1}dz_{2}d\gamma)=\eta(dz_{1}dz_{2}|\gamma)\pi(d\gamma).
\]

Hence, by the definition of viability, we obtain for every $\tilde{f}%
\in\mathcal{D}(\tilde{L})$ that
\[
\int_{\Gamma}\tilde{L}\tilde{f}(\gamma)\pi(d\gamma)=0
\]
where we used the independence of $\tilde{L}$ on $z$ to eliminate the
stochastic kernel $\eta$. Then Proposition \ref{P:NewMeasureRandomCase}
guarantees that $\pi$ takes the form
\[
\pi(d\gamma)\doteq\frac{\tilde{m}(\gamma)}{\mathrm{E}^{\nu}\tilde{m}(\cdot
)}\nu(d\gamma)
\]
and is actually an invariant, ergodic and reversible probability measure for
the process associated with the operator $\tilde{L}$, or equivalently for the
environment process $\gamma_{t}$ as given by (\ref{Eq:EnvironmentProcess}).
Next, since the cost $\left\Vert z\right\Vert^{2}$ is convex in $z=(z_{1},z_{2})$ and $\tilde{\lambda}_{\rho}$
is affine in $z$, the relaxed control formulation can be easily written in
terms of the ordinary control formulation%

\begin{equation}
L^{o}(x,v)=\inf_{\tilde{u}\in\mathcal{A}_{x,v}^{o}}\frac{1}{2}\mathrm{E}^{\pi
}\left[\left\Vert \tilde{u}_{1}(\cdot)\right\Vert^{2}+\left\Vert \tilde{u}_{2}(\cdot)\right\Vert^{2}\right] \label{Eq:LocalRateFunction_OrdinaryFormulation}%
\end{equation}
and
\[
\mathcal{A}_{x,v}^{o}=\left\{  \tilde{u}=(\tilde{u}_{1},\tilde{u}_{2}):\Gamma\mapsto\mathbb{R}%
^{d}:\mathrm{E}^{\pi}\left[\left\Vert \tilde{u}_{1}(\cdot)\right\Vert^{2}+\left\Vert \tilde{u}_{2}(\cdot)\right\Vert^{2}\right] <\infty,\text{ and }v=\lim
_{\rho\rightarrow0}\mathrm{E}^{\pi}\left[  \tilde{\lambda}_{\rho}%
(x,\cdot,\tilde{u}_{1}(\cdot),\tilde{u}_{2}(\cdot))\right]  \right\}  .
\]

Jensen's inequality and the fact that $\tilde{\lambda}_{\rho}(x,\gamma,z_{1},z_{2})$ is
affine in $z$ imply $L^{r}(x,v)\geq L^{o}(x,v)$. For the reverse inequality,
for given $\tilde{u}=(\tilde{u}_{1},\tilde{u}_{2})$ one can define a corresponding relaxed control by
$\mathrm{P}(dz_{1}dz_{2}d\gamma)=\delta_{\left(\tilde{u}_{1}(\gamma),\tilde{u}_{2}(\gamma)\right)}(dz_{1}dz_{2})\pi(d\gamma)$.

The next step is to prove that the infimization problem in
(\ref{Eq:LocalRateFunction_OrdinaryFormulation}) can be solved explicitly and
in particular that
\begin{equation}
L^{o}(x,v)=\frac{1}{2}(v-r(x))^{T}q^{-1}(x)(v-r(x)) \label{Eq:LocalRateFunction}%
\end{equation}
where%
\[
r(x)=\lim_{\rho\downarrow0}\mathrm{E}^{\pi}\left[  \tilde{c}(x,\cdot) +D\tilde{\chi}_{\rho
}(\cdot) \tilde{g}(x,\cdot)\right]  =\mathrm{E}^{\pi}[\tilde{c}(x,\cdot)+\tilde{\xi}%
(\cdot)\tilde{g}(x,\cdot)]
\]%
\begin{align*}
q(x)&=\lim_{\rho\downarrow0}\mathrm{E}^{\pi}\left[  (\tilde{\sigma}(x,\cdot)+D\tilde{\chi}_{\rho}%
(\cdot)\tilde{\tau}_{1}(\cdot))(\tilde{\sigma}(x,\cdot)+D\tilde{\chi}_{\rho}%
(\cdot)\tilde{\tau}_{1}(\cdot))^{T}+\left(D\tilde{\chi}_{\rho}%
(\cdot)\tilde{\tau}_{2}(\cdot)\right)\left(D\tilde{\chi}_{\rho}%
(\cdot)\tilde{\tau}_{2}(\cdot)\right)^{T}\right]\nonumber\\
&=\mathrm{E}^{\pi}\left[  (\tilde{\sigma}(x,\cdot)+\tilde{\xi}(\cdot)\tilde{\tau}_{1}(\cdot))(\tilde{\sigma}(x,\cdot)+\tilde{\xi}
(\cdot)\tilde{\tau}_{1}(\cdot))^{T}+\left(\tilde{\xi}
(\cdot)\tilde{\tau}_{2}(\cdot)\right)\left(\tilde{\xi}
(\cdot)\tilde{\tau}_{2}(\cdot)\right)^{T}\right]
\end{align*}

Let us first prove that for every $\tilde{u}=\left(\tilde{u}_{1},\tilde{u}_{2}\right)\in\mathcal{A}%
_{x,v}^{o}$
\begin{equation}
\mathrm{E}^{\pi}\left\Vert \tilde{u}(x,\cdot)\right\Vert^{2}\geq(v-r(x))^{T}q^{-1}(x)(v-r(x)).
\label{Eq:cost_bound}%
\end{equation}

By definition, any $\tilde{u}=\left(\tilde{u}_{1},\tilde{u}_{2}\right)\in\mathcal{A}_{x,v}^{o}$ satisfies
\begin{align}
v&=\lim_{\rho\rightarrow0}\mathrm{E}^{\pi}\left[  \tilde{\lambda}_{\rho
}(x,\cdot,\tilde{u}_{1}(\cdot),\tilde{u}_{1}(\cdot))\right]\nonumber\\
&=r(x)+\lim_{\rho\rightarrow
0}\mathrm{E}^{\pi}\left[ \left( \tilde{\sigma}(x,\cdot)+D\tilde{\chi}_{\rho}(\cdot)\tilde{\tau}_{1}(\cdot)\right)\tilde{u}_{1}(\cdot)+
D\tilde{\chi}_{\rho}(\cdot)\tilde{\tau}_{2}(\cdot)\tilde{u}_{2}(\cdot)\right]  .\nonumber
\end{align}

Treating $x$ as a parameter, define
\[
\hat{v}=v-r(x)=\lim_{\rho\rightarrow
0}\mathrm{E}^{\pi}\left[ \left( \tilde{\sigma}(x,\cdot)+D\tilde{\chi}_{\rho}(\cdot)\tilde{\tau}_{1}(\cdot)\right)\tilde{u}_{1}(\cdot)+
D\tilde{\chi}_{\rho}(\cdot)\tilde{\tau}_{2}(\cdot)\tilde{u}_{2}(\cdot)\right],
\]
and for notational convenience set
\[
\tilde{\kappa}_{1,\rho}(x,\gamma)=\tilde{\sigma}(x,\gamma)+D\tilde{\chi}_{\rho}(\gamma)\tilde{\tau}_{1}(\gamma)\quad\textrm{ and }
\tilde{\kappa}_{2,\rho}(x,\gamma)=D\tilde{\chi}_{\rho}(\gamma)\tilde{\tau}_{2}(\gamma)
\]

Next, we drop writing explicitly the dependence on the parameter $x$ and we
write $q^{-1}=W^{T}W$, where $W$ is an invertible matrix, so that $\hat{v}%
^{T}q^{-1}\hat{v}=\left\Vert W\hat{v}\right\Vert^{2}$.  Without loss of generality, we assume that
$\tilde{u}\in L^{2}(\Gamma)$ is such that $\mathrm{E}^{\pi}\left\Vert \tilde{u}%
(\cdot)\right\Vert^{2}=1$. By
Cauchy-Schwartz inequality in $\mathbb{R}^{m}$ we have
\begin{align}
\left\Vert W\hat{v}\right\Vert^{2}  &  =\left\langle W\hat{v},W\lim_{\rho\downarrow0}%
\mathrm{E}^{\pi}\left[  \tilde{\kappa}_{1,\rho}(\cdot)\tilde{u}_{1}(\cdot)+ \tilde{\kappa}_{2,\rho}(\cdot)\tilde{u}_{2}(\cdot)\right]
\right\rangle \nonumber\\
&  =\lim_{\rho\downarrow0}\mathrm{E}^{\pi}\left[  \left\langle \tilde{u}_{1}
(\cdot),\tilde{\kappa}_{1,\rho}^{T}(\cdot)W^{T}W\hat{v}\right\rangle + \left\langle \tilde{u}_{2}
(\cdot),\tilde{\kappa}_{2,\rho}^{T}(\cdot)W^{T}W\hat{v}\right\rangle\right]
\nonumber\\
&  \leq\lim_{\rho\downarrow0}\left(  \mathrm{E}^{\pi}\left[\left\Vert \tilde{\kappa
}_{1,\rho}^{T}(\cdot)W^{T}W\hat{v}\right\Vert^{2}+\left\Vert \tilde{\kappa
}_{2,\rho}^{T}(\cdot)W^{T}W\hat{v}\right\Vert^{2}\right] \right)  ^{1/2}\nonumber\\
&  =\lim_{\rho\downarrow0}\left(  \hat{v}^{T}W^{T}W\mathrm{E}^{\pi}\left[
\tilde{\kappa}_{1,\rho}(\cdot)\tilde{\kappa}_{1,\rho}^{T}(\cdot)+
\tilde{\kappa}_{2,\rho}(\cdot)\tilde{\kappa}_{2,\rho}^{T}(\cdot)\right]
W^{T}W\hat{v}\right)  ^{1/2}\nonumber\\
&  =\left(  \hat{v}^{T}W^{T}WqW^{T}W\hat{v}\right)  ^{1/2}\nonumber\\
&  =\left\Vert W\hat{v}\right\Vert.\nonumber
\end{align}

If $\left\Vert W\hat{v}\right\Vert=0$, then (\ref{Eq:cost_bound}) holds automatically. If
$\left\Vert W\hat{v}\right\Vert\neq0$, then the last display implies $\left\Vert W\hat{v}\right\Vert\leq1$, which
directly proves that%

\[
\mathrm{E}^{\pi}\left\Vert\tilde{u}(x,\cdot)\right\Vert^{2}=1\geq\left\Vert W\hat{v}\right\Vert^{2}=(v-r(x))^{T}%
q^{-1}(x)(v-r(x)).
\]

To prove that the inequality becomes an equality when taking the infimum over
all $\tilde{u}\in\mathcal{A}_{x,v}^{o}$, we need to find a $\tilde{u}\in
L^{2}(\Gamma)$ which attains the infimum. Define the elements of $L^{2}(\Gamma)$%

\[
\tilde{u}_{1,\rho}(x,\gamma;v)=\left(\tilde{\sigma}(\gamma)+ D\tilde
{\chi}_{\rho}(\gamma)\tilde{\tau}_{1}(\gamma)\right)^{T}q^{-1}(x)(v-r(x))
\]
and
\[
\tilde{u}_{2,\rho}(x,\gamma;v)=\left(D\tilde
{\chi}_{\rho}(\gamma)\tilde{\tau}_{2}(\gamma)\right)^{T}q^{-1}(x)(v-r(x))
\]
and set $\tilde{u}_{\rho}(x,\cdot;v)=\left(\tilde{u}_{1,\rho}(x,\cdot;v),\tilde{u}_{2,\rho}(x,\cdot;v)\right)$. A straightforward computation yields
\begin{align*}
\mathrm{E}^{\pi}\left\Vert \tilde{u}_{\rho}(x,\cdot;v)\right\Vert^{2}&=(v-r(x))^{T}q^{-1}(x)
\mathrm{E}^{\pi}\left[  (\tilde{\sigma}(x,\cdot)+D\tilde{\chi}_{\rho}%
(\cdot)\tilde{\tau}_{1}(\cdot))(\tilde{\sigma}(x,\cdot)+D\tilde{\chi}_{\rho}%
(\cdot)\tilde{\tau}_{1}(\cdot))^{T}+\right.\nonumber\\
&\hspace{6cm}\left.+\left(D\tilde{\chi}_{\rho}%
(\cdot)\tilde{\tau}_{2}(\cdot)\right)\left(D\tilde{\chi}_{\rho}%
(\cdot)\tilde{\tau}_{2}(\cdot)\right)^{T}\right]
q(x)(v-r(x))
\end{align*}

Thus, letting $\rho\downarrow0$, we obtain
\[
\lim_{\rho\downarrow0}\mathrm{E}^{\pi}\left\Vert\tilde{u}_{\rho}(x,\cdot;v)\right\Vert^{2}%
=(v-r(x))^{T}q^{-1}(x)(v-r(x))
\]

Hence, the element $\tilde{u}\in L^{2}(\Gamma)$ that we are looking for is the
$L^{2}(\pi)$ limit of $\tilde{u}_{\rho}$ as defined above. This is well
defined, since by Proposition 2.6 in \cite{Olla1994} $D\tilde{\chi}_{\rho}$
has a well defined $L^{2}(\pi)$ strong limit. Therefore, we have proven that
\[
L^{o}(x,v)=\frac{1}{2}(v-r(x))^{T}q^{-1}(x)(v-r(x))
\]
which concludes the proof of Theorem \ref{T:MainTheorem3}.

\appendix
\section{Quenched ergodic theorems }\label{S:QuenchedErgodicTheorems}
In this appendix we prove quenched ergodic theorems that are required for the proof of Theorem \ref{T:MainTheorem1}. For notational convenience and without loss of generality, we mostly consider a process $Y$ driven by a single
Brownian motion with diffusion coefficient $\kappa(y,\gamma)$ such that $\kappa\kappa^{T}=\tau_{1}\tau_{1}^{T}+\tau_{2}\tau_{2}^{T}$.

We prove the required ergodic result, Lemma \ref{L:ErgodicTheorem4} in a progressive  way. First, in Lemma \ref{L:ErgodicTheorem0} we recall the classical ergodic theorem. This is strengthened in  Lemma \ref{L:ErgodicTheorem1u} to cover cases of time shifts, uniformly with respect to $t\in[0,T]$. Then, in Lemmas \ref{L:ErgodicTheorem2}-\ref{L:ErgodicTheorem3} we consider the case of perturbing the drift of the process by small perturbations (uncontrolled and controlled case). The latter result together with the standard technique of freezing the slow component yield the proof of the ergodic statement in Lemma \ref{L:ErgodicTheorem4}.

\subsection{No time shifts, i.e. $t=0$}

\begin{lemma}
\label{L:ErgodicTheorem0} Consider the process $Y_{t}^{\epsilon,y_{0},\gamma}$
satisfying the SDE
\begin{equation}
Y_{t}^{\epsilon,y_{0},\gamma}=y_{0}+\frac{\epsilon}{\delta^{2}}\int_{0}%
^{t}f(Y_{s}^{\epsilon,y_{0},\gamma},\gamma)ds+\frac{\sqrt{\epsilon}}{\delta
}\int_{0}^{t}\kappa(Y_{s}^{\epsilon,y_{0},\gamma},\gamma)dW_{s}.
\label{Eq:BasicSDE}%
\end{equation}

Consider also a function $\tilde{\Psi}\in L^{2}(\Gamma)\cap L^{1}(\pi
)$ and define $\Psi
(y,\gamma)=\tilde{\Psi}(\tau_{y}\gamma)$. Assume that $\Psi:\mathbb{R}%
^{d-m}\times\Gamma\mapsto\mathbb{R}$ is measurable.

Denote $\bar{\Psi}\doteq\int_{\Gamma}\tilde{\Psi}(\gamma)\pi(d\gamma)$. Then
for any sequence $h(\epsilon)$ that is bounded from above and such that
$\delta^{2}/[\epsilon h(\epsilon)]\downarrow0$ (note that in particular
$h(\epsilon)$ could be a constant), there is a set $N$ of full $\pi-$measure
such that for every $\gamma\in N$
\[
\lim_{\epsilon\downarrow0}\mathbb{E}\left\vert \frac{1}{h(\epsilon
)}\int_{0}^{h(\epsilon)}\Psi(Y_{s}^{\epsilon,y_{0},\gamma}%
,\gamma)ds-\bar{\Psi}\right\vert =0.
\]
\end{lemma}

\begin{proof}
[Proof of Lemma \ref{L:ErgodicTheorem0}]Let $\hat{Y}_{t}^{y_{0},\gamma
}=Y_{\delta^{2}t/\epsilon}^{\epsilon,y_{0},\gamma}$. Note that $\hat{Y}%
_{t}^{y_{0},\gamma}$ satisfies
\begin{equation}
\hat{Y}_{t}^{y_{0},\gamma}=y_{0}+\int_{0}^{t}f(\hat{Y}_{s}^{y_{0},\gamma
},\gamma)ds+\int_{0}^{t}\kappa(\hat{Y}_{s}^{y_{0},\gamma}%
,\gamma)dW_{s}, \label{Eq:Yhat}%
\end{equation}
and also that $\pi(d\gamma)$ is the invariant ergodic probability measure for
the environment process $\gamma_{t}=\tau_{\hat{Y}_{t}^{y_{0},\gamma}}\gamma$
(Proposition \ref{P:NewMeasureRandomCase}).

Suppose that $\delta^{2}/[\epsilon h(\epsilon)]\downarrow0$. By the ergodic
theorem, there is a set $N$ of full $\pi-$measure such that for any $\gamma\in
N$
\begin{align*}
\lim_{\epsilon\downarrow0}\mathbb{E}\left[  \frac{1}{h(\epsilon)}%
\int_{0}^{h(\epsilon)}\Psi(Y_{s}^{\epsilon,y_{0},\gamma},\gamma)ds\right]   &
=\lim_{\epsilon\downarrow0}\mathbb{E}\left[  \frac{\delta^{2}%
}{\epsilon h(\epsilon)}\int_{0}^{\frac{\epsilon h(\epsilon)}{\delta^{2}}}%
\Psi(\hat{Y}_{s}^{y_{0},\gamma},\gamma)ds\right] \\
&  \qquad=\lim_{\epsilon\downarrow0}\mathbb{E}\left[  \frac
{\delta^{2}}{\epsilon h(\epsilon)}\int_{0}^{\frac{\epsilon h(\epsilon)}%
{\delta^{2}}}\tilde{\Psi}(\tau_{\hat{Y}_{s}^{y_{0},\gamma}}\gamma)ds\right] \\
&  \qquad=\lim_{\epsilon\downarrow0}\mathbb{E}\left[  \frac
{\delta^{2}}{\epsilon h(\epsilon)}\int_{0}^{\frac{\epsilon h(\epsilon)}%
{\delta^{2}}}\tilde{\Psi}(\gamma_{s})ds\right] \\
&  \qquad=\bar{\Psi}.
\end{align*}

\end{proof}

It follows from Egoroff's theorem that for every $\eta>0$ there is a set
$N_{\eta}$ with $\pi\left[  N_{\eta}\right]  >1-\eta$, such that
\[
\lim_{\epsilon\downarrow0}\sup_{\gamma\in N_{\eta}}\mathbb{E}\left\vert \frac{1}{h(\epsilon)}\int_{0}^{h(\epsilon)}\Psi(Y_{s}%
^{\epsilon,y_{0},\gamma})ds-\bar{\Psi}\right\vert =0.
\]

\subsection{Time shifts and uniformity}

For notational purposes we will write that $h(\epsilon)\in \mathcal{H}_{1}^{N_{\eta}}$, if the pair $(h(\epsilon),N_{\eta})$ satisfies Condition \ref{A:h_functionUniform}.
\begin{condition}
\label{A:h_functionUniform} Let $\tilde{\Psi}\in L^{2}(\Gamma)\cap L^{1}(\pi
)$ and define the measurable function $\Psi
(y,\gamma)=\tilde{\Psi}(\tau_{y}\gamma)$.  For $\gamma\in\Gamma$ define
\[
\theta^{\gamma}(u)=\sup_{r>u}\mathbb{E}\left|  \frac{1}{r}\int_{0}^{r}%
\Psi(\hat{Y}_{s}^{y_{0},\gamma},\gamma)ds-\bar{\Psi}\right|  .
\]
For any $\eta\in(0,1)$, there exists a set $N_{\eta}$ with $\pi(N_{\eta})\geq1-\eta$ and a sequence $\{h(\epsilon),\epsilon>0\}$
such that the following are satisfied:

\begin{enumerate}
\item {$\frac{\delta^{2}/\epsilon}{h(\epsilon)}\rightarrow0$ as $\epsilon
\downarrow0$, }

\item {there exists  $\beta\in(0,1)$ such that $\frac{\sup_{\gamma\in N_{\eta}} \theta^{\gamma}\left(
\frac{1}{\left(  \delta^{2}/\epsilon\right)  ^{\beta}}\right)  }{h(\epsilon
)}\rightarrow0$, as $\epsilon\downarrow0$, and}

\item {$\frac{1}{h(\epsilon)}\sup_{\gamma\in N_{\eta}}\sup_{t\in[0,T]}\mathbb{E}\left|
\left(  \delta^{2}/\epsilon\right)  \int^{\frac{t}{\delta^{2}/\epsilon}}%
_{0}\Psi(\hat{Y}_{s}^{y_{0},\gamma},\gamma)ds-t \bar{\Psi}\right|
\rightarrow0$ as $\epsilon\downarrow0$}
\end{enumerate}
\end{condition}

Lemma \ref{L:ErgodicTheorem1u} shows that one in fact can find a pair $(h(\epsilon),N_{\eta})$ satisfies Condition \ref{A:h_functionUniform} in order to prove a uniform
in time $t\in[0,T]$, ergodic theorem.
\begin{lemma}
\label{L:ErgodicTheorem1u} Consider the setup and notations of Lemma
\ref{L:ErgodicTheorem0}. Fix  $\eta>0$.  Then there exists
a set $N_{\eta}$ such that $\pi(N_{\eta})\geq 1-\eta$ and $h(\epsilon)\in \mathcal{H}_{1}^{N_{\eta}} $ such that
\begin{equation}
\lim_{\epsilon\downarrow0}\sup_{\gamma\in N_{\eta}}\sup_{t\in[0,T]}\mathbb{E}\left\vert
\frac{1}{h(\epsilon)}\int_{t}^{t+h(\epsilon)}\Psi(Y_{s}^{\epsilon,y_{0}
,\gamma},\gamma)ds-\bar{\Psi}\right\vert =0.\label{Eq:ErgodicTheoremUniform}
\end{equation}
\end{lemma}

\begin{proof}
[Proof of Lemma \ref{L:ErgodicTheorem1u}] We start with the following decomposition%

\begin{align}
&  \mathbb{E}\left|  \frac{1}{h(\epsilon)}\int_{t}^{t+h(\epsilon
)}\Psi(Y_{s}^{\epsilon,y_{0},\gamma},\gamma)ds-\bar{\Psi}\right|  =
\mathbb{E}\left|  \frac{\delta^{2}/\epsilon}{h(\epsilon)}\int
_{\frac{t}{\delta^{2}/\epsilon}}^{\frac{t+h(\epsilon)}{\delta^{2}/\epsilon}%
}\Psi(\hat{Y}_{s}^{y_{0},\gamma},\gamma)ds-\bar{\Psi}\right|
\nonumber\\
&  \quad=\mathbb{E}\left|  \frac{\delta^{2}/\epsilon}{h(\epsilon
)}\int_{0}^{\frac{t+h(\epsilon)}{\delta^{2}/\epsilon}}\Psi(\hat{Y}_{s}%
^{y_{0},\gamma},\gamma)ds-\frac{\delta^{2}/\epsilon}{h(\epsilon)}%
\int^{\frac{t}{\delta^{2}/\epsilon}}_{0}\Psi(\hat{Y}_{s}^{y_{0},\gamma
},\gamma)ds-\bar{\Psi}\right| \nonumber\\
&  \quad=\mathbb{E}\left|  \frac{t+h(\epsilon)}{h(\epsilon)} \left(
\frac{\delta^{2}/\epsilon}{t+h(\epsilon)}\int_{0}^{\frac{t+h(\epsilon)}%
{\delta^{2}/\epsilon}}\Psi(\hat{Y}_{s}^{y_{0},\gamma},\gamma
)ds-\bar{\Psi}\right)  -\frac{t}{h(\epsilon)}\left(  \frac{\delta^{2}%
/\epsilon}{t}\int^{\frac{t}{\delta^{2}/\epsilon}}_{0}\Psi(\hat{Y}_{s}%
^{y_{0},\gamma},\gamma)ds-\bar{\Psi}\right)  \right| \nonumber\\
&  \quad\leq\frac{T+h(\epsilon)}{h(\epsilon)} \mathbb{E}\left|
\frac{\delta^{2}/\epsilon}{t+h(\epsilon)}\int_{0}^{\frac{t+h(\epsilon)}%
{\delta^{2}/\epsilon}}\Psi(\hat{Y}_{s}^{y_{0},\gamma},\gamma
)ds-\bar{\Psi}\right|  + \frac{1}{h(\epsilon)}\mathbb{E}\left|
\delta^{2}/\epsilon\int^{\frac{t}{\delta^{2}/\epsilon}}_{0}\Psi(\hat{Y}%
_{s}^{y_{0},\gamma},\gamma)ds-t \bar{\Psi}\right| \nonumber\\
&  \quad\leq\frac{T+h(\epsilon)}{h(\epsilon)} \sup_{r>\frac{t+h(\epsilon
)}{\delta^{2}/\epsilon}}\mathbb{E}\left|  \frac{1}{r}\int_{0}^{r}%
\Psi(\hat{Y}_{s}^{y_{0},\gamma},\gamma)ds-\bar{\Psi}\right|  + \frac
{1}{h(\epsilon)}\mathbb{E}\left|  \delta^{2}/\epsilon\int^{\frac
{t}{\delta^{2}/\epsilon}}_{0}\Psi(\hat{Y}_{s}^{y_{0},\gamma}%
,\gamma)ds-t \bar{\Psi}\right| \nonumber\\
&  \quad\leq\frac{T+1}{h(\epsilon)} \theta^{\gamma}\left(  \frac{t+h(\epsilon
)}{\delta^{2}/\epsilon}\right)  + \frac{1}{h(\epsilon)}\mathbb{E}^{\gamma
}\left|  \left(  \delta^{2}/\epsilon\right)  \int^{\frac{t}{\delta
^{2}/\epsilon}}_{0}\Psi(\hat{Y}_{s}^{y_{0},\gamma},\gamma)ds-t
\bar{\Psi}\right|\nonumber
\end{align}
by choosing $h(\epsilon)<1$ and defining
\[
\theta^{\gamma}(u)=\sup_{r>u}\mathbb{E}\left|  \frac{1}{r}\int_{0}^{r}%
\Psi(\hat{Y}_{s}^{y_{0},\gamma},\gamma)ds-\bar{\Psi}\right|  .
\]

Thus, we have proven that
\begin{align}
&  \sup_{t\in[0,T]}\mathbb{E}\left|  \frac{1}{h(\epsilon)}\int
_{t}^{t+h(\epsilon)}\Psi(Y_{s}^{\epsilon,y_{0},\gamma},\gamma
)ds-\bar{\Psi}\right|  \leq\nonumber\\
&  \qquad\leq\frac{T+1}{h(\epsilon)} \sup_{t\in[0,T]}\theta^{\gamma}\left(
\frac{t+h(\epsilon)}{\delta^{2}/\epsilon}\right)  +\frac{1}{h(\epsilon)}%
\sup_{t\in[0,T]}\mathbb{E}\left|  \left(  \delta^{2}/\epsilon\right)
\int^{\frac{t}{\delta^{2}/\epsilon}}_{0}\Psi(\hat{Y}_{s}^{y_{0},\gamma
},\gamma)ds-t \bar{\Psi}\right|  \label{Eq:BoundForUniformStatement}%
\end{align}

Let us first treat the second term on the right hand side of
(\ref{Eq:BoundForUniformStatement}). By the ergodic theorem, Lemma
\ref{L:ErgodicTheorem0}, and Egoroff's theorem we know that there exists a set $N_{\eta}$ with $\pi(N_{\eta})\geq 1-\eta$ such that
\begin{equation*}
\lim_{\epsilon\downarrow0}\sup_{\gamma\in N_{\eta}}\sup_{t\in[0,T]}\mathbb{E}\left|  \left(
\delta^{2}/\epsilon\right)  \int^{\frac{t}{\delta^{2}/\epsilon}}_{0}\Psi
(\hat{Y}_{s}^{y_{0},\gamma},\gamma)ds-t \bar{\Psi}\right|  =0
\end{equation*}

So, if we choose $h(\epsilon)\downarrow0$ such that
\begin{equation*}
\lim_{\epsilon\downarrow0}\frac{1}{h(\epsilon)}\sup_{\gamma\in N_{\eta}}\sup_{t\in[0,T]}\mathbb{E}%
^{\gamma}\left|  \left(  \delta^{2}/\epsilon\right)  \int^{\frac{t}{\delta
^{2}/\epsilon}}_{0}\Psi(\hat{Y}_{s}^{y_{0},\gamma},\gamma)ds-t
\bar{\Psi}\right|  =0
\end{equation*}
we have that the second term on the right hand side of
(\ref{Eq:BoundForUniformStatement}) goes to zero. Next, we treat the first
term on the right hand side of (\ref{Eq:BoundForUniformStatement}). Since, the
function $\theta^{\gamma}(u)$ is decreasing, we get that
\[
\theta^{\gamma}\left(  \frac{t+h(\epsilon)}{\delta^{2}/\epsilon}\right)  \leq
\theta^{\gamma}\left(  \frac{h(\epsilon)}{\delta^{2}/\epsilon}\right)
\]

Thus, we have obtained that for every $\gamma\in \Gamma$
\begin{equation*}
\sup_{t\in[0,T]}\theta^{\gamma}\left(  \frac{t+h(\epsilon)}{\delta^{2}/\epsilon
}\right)  \leq\sup_{t\in[0,T]}\theta^{\gamma}\left(  \frac{h(\epsilon)}{\delta
^{2}/\epsilon}\right)  =\theta^{\gamma}\left(  \frac{h(\epsilon)}{\delta
^{2}/\epsilon}\right)
\end{equation*}
Notice that because $h(\epsilon)$ is chosen such that $\frac{\delta
^{2}/\epsilon}{h(\epsilon)}\downarrow0$, Lemma \ref{L:ErgodicTheorem0} and Egoroff's theorem,
imply that
\[
\lim_{\epsilon\downarrow0}\sup_{\gamma\in N_{\eta}}\theta^{\gamma}\left(  \frac{h(\epsilon)}{\delta
^{2}/\epsilon}\right)  =0
\]

Therefore, the first term on the right hand side of
(\ref{Eq:BoundForUniformStatement}) goes to zero, if we can choose
$h(\epsilon)$, such that $\sup_{\gamma\in N_{\eta}}\theta^{\gamma}\left(  \frac{h(\epsilon)}{\delta
^{2}/\epsilon}\right)  /h(\epsilon)\downarrow0$. This is a little bit tricky
here because the argument of $\theta$ depends on $h(\epsilon)$. However, this can be done as
follows. Fix $\beta\in(0,1)$ (e.g., $\beta=1/2$) and choose $h(\epsilon
)\geq\left(  \delta^{2}/\epsilon\right)  ^{1-\beta}$. Then, the monotonicity
of $f$, implies that
\[
\theta^{\gamma}\left(  \frac{h(\epsilon)}{\delta^{2}/\epsilon}\right)  \leq
\theta^{\gamma}\left(  \frac{1}{\left(  \delta^{2}/\epsilon\right)  ^{\beta}%
}\right)  \downarrow0
\]

This proves that we can choose $h(\epsilon)$ such that the first term of the
right hand of (\ref{Eq:BoundForUniformStatement}) goes to zero. The claim
follows, by noticing that the previous computations imply that we can choose
$h(\epsilon)$ that may go to zero, but slowly enough, such that both the first
and the second term on the right hand side of
(\ref{Eq:BoundForUniformStatement}) go to zero.
\end{proof}

\subsection{Ergodic theorems with perturbation by small drift-Uncontrolled case}

\begin{lemma}
\label{L:ErgodicTheorem2} Consider the process $Y_{t}^{\epsilon,y_{0},\gamma}$
satisfying the SDE
\[
Y_{t}^{\epsilon,x,y_{0},\gamma}=y_{0}+\frac{\epsilon}{\delta^{2}}\int_{0}%
^{t}f(Y_{s}^{\epsilon,x,y_{0},\gamma},\gamma)ds+\frac{1}{\delta}\int_{0}%
^{t}g(s,x,Y_{s}^{\epsilon,x,y_{0},\gamma},\gamma)ds+\frac{\sqrt{\epsilon}%
}{\delta}\int_{0}^{t}\kappa(Y_{s}^{\epsilon,x,y_{0},\gamma}%
,\gamma)dW_{s}%
\]

Let us consider a function $\tilde{\Psi}:[0,T]\times\mathbb{R}^{m}\times
\Gamma$ such that $\tilde{\Psi}(t,x,\cdot)\in L^{2}(\Gamma)\cap L^{1}(\pi
)$ and define $\Psi(t,x,y,\gamma)=\tilde{\Psi
}(t,x,\tau_{y}\gamma)$. We assume that the function $\Psi:[0,T]\times
\mathbb{R}^{m}\times\mathbb{R}^{d-m}\times\Gamma\mapsto\mathbb{R}$ is
measurable, piecewise constant in $t$ and uniformly continuous in $x$ with
respect to $(t,y)$.

Denote $\bar{\Psi}(t,x)\doteq\int_{\Gamma}\tilde{\Psi}(t,x,\gamma)\pi
(d\gamma)$ for all $(t,x)\in[0,T]\times\mathbb{R}^{m}$. Fix  $\eta>0$.  Then there exists
a set $N_{\eta}$ such that $\pi(N_{\eta})\geq 1-\eta$ and $h(\epsilon)\in \mathcal{H}_{1}^{N_{\eta}} $ such that
\begin{equation*}
\lim_{\epsilon\downarrow0}\sup_{\gamma\in N_{\eta}}\sup_{0\leq t\leq T}\mathbb{E}\left\vert
\frac{1}{h(\epsilon)}\int_{t}^{t+h(\epsilon)}\Psi(s,x,Y_{s}^{\epsilon,x,y_{0},\gamma},\gamma)ds-\bar{\Psi}(t,x)\right\vert =0
\end{equation*}
locally uniformly with respect to the parameter $x\in\mathbb{R}^{m}$.
\end{lemma}

\begin{proof}
[Proof of Lemma \ref{L:ErgodicTheorem2}]Let us set $\hat{Y}_{t}^{\epsilon, x,y_{0}%
,\gamma}=Y_{\delta^{2}t/\epsilon}^{\epsilon,x,y_{0},\gamma}$. Notice
that $\hat{Y}_{t}^{\epsilon,x,y_{0},\gamma}$ satisfies
\[
\hat{Y}_{t}^{\epsilon,x,y_{0},\gamma}=y_{0}+\int_{0}^{t}f(\hat{Y}_{s}%
^{\epsilon,x,y_{0},\gamma},\gamma)ds+\frac{\delta}{\epsilon}\int_{0}%
^{t}g\left(  \frac{\delta^{2}}{\epsilon}s,x,\hat{Y}_{s}^{\epsilon, x,y_{0}%
,\gamma},\gamma\right)  ds+\int_{0}^{t}\kappa(\hat{Y}%
_{s}^{\epsilon,x,y_{0},\gamma},\gamma)dW_{s}%
\]

Slightly abusing notation, we denote by $Y_{t}^{\epsilon,y_{0},\gamma}$ and
$\hat{Y}^{y_{0},\gamma}_{t}$ the processes corresponding to $Y_{t}%
^{\epsilon,x,y_{0},\gamma}$ and $\hat{Y}^{\epsilon,x,y_{0},\gamma}_{t}$ with
$c(t,x,y)=0$.

Lemma \ref{L:ErgodicTheorem1u} guarantees that the statement of the Lemma is
true for $Y_{t}^{\epsilon,y_{0},\gamma}$, namely that there exists
a set $N_{\eta}$ such that $\pi(N_{\eta})\geq 1-\eta$ and $h(\epsilon)\in \mathcal{H}_{1}^{N_{\eta}} $ such that
\begin{equation}
\lim_{\epsilon\downarrow0}\sup_{\gamma\in N_{\eta}}\sup_{t\in[0,T]}\mathbb{E}\left\vert
\frac{1}{h(\epsilon)}\int_{t}^{t+h(\epsilon)}\Psi(s,x,Y_{s}^{\epsilon,y_{0}%
,\gamma},\gamma)ds-\bar{\Psi}(t,x)\right\vert =0.
\label{Eq:ErgodicTheorem1}%
\end{equation}

 The fact that the convergence is
also locally uniform with respect to the parameter $x\in\mathbb{R}^{m}$
follows by the uniform continuity of $\Psi$ in $x$. This implies that in Lemma
\ref{L:ErgodicTheorem1u}, we can choose the sequence $h(\epsilon)$ so that the
convergence holds uniformly with respect to $x$ in each bounded region,
see for example Theorem II.3.11 in \cite{Skorokhod1987}.

To translate this statement to what we need we use Girsanov's theorem on the
absolutely continuous change of measures on the space of trajectories in
$C([0,T];\mathbb{R}^{d-m})$. Let
\[
\phi(s,x,y,\gamma)=-\kappa^{-1}(y,\gamma)g(s,x,y,\gamma)
\]
and define the quantity
\begin{equation*}
M^{\epsilon,\gamma}_{T}=e^{\frac{\delta}{\epsilon}\frac{1}{\sqrt{2}} \int
_{0}^{T}\phi(\delta^{2}s/\epsilon,x,\hat{Y}^{y_{0},\gamma}_{s}%
,\gamma)dW_{s}-\frac{1}{2}\left(  \frac{\delta}{\epsilon}\right)  ^{2}\frac
{1}{2}\int_{0}^{T}\left\Vert  \phi(\delta^{2}s/\epsilon,x,\hat{Y}^{y_{0},\gamma}_{s},\gamma)\right\Vert  ^{2}ds}%
\end{equation*}

Then, by the aforementioned Girsanov's theorem, for each$\gamma\in\Gamma$,
$M^{\epsilon,\gamma}_{T}$ is a $\mathbb{P}^{\gamma}$ martingale.
Therefore, we obtain
\begin{align}
\mathbb{E}\frac{1}{h(\epsilon)}\int_{t}^{t+h(\epsilon)}%
\Psi(s,x,Y^{\epsilon,x,y_{0},\gamma}_{s},\gamma)ds  &  = \mathbb{E}%
\frac{\delta^{2}/\epsilon}{h(\epsilon)}\int_{\frac{t}{\delta
^{2}/\epsilon}}^{\frac{t+h(\epsilon)}{\delta^{2}/\epsilon}}\Psi(s,x,\hat
{Y}^{\epsilon,x,y_{0},\gamma}_{s},\gamma)ds\nonumber\\
&  =\mathbb{E}\left[  \left(  \frac{\delta^{2}/\epsilon}{h(\epsilon
)}\int_{\frac{t}{\delta^{2}/\epsilon}}^{\frac{t+h(\epsilon)}{\delta
^{2}/\epsilon}}\Psi(s,x,\hat{Y}^{y_{0},\gamma}_{s},\gamma)ds\right)
M^{\epsilon,\gamma}_{T}\right]\nonumber
\end{align}

Next, we prove that, for every $\gamma\in\Gamma$, $M^{\epsilon,\gamma}_{T}$
converges to $1$ in probability as $\epsilon\downarrow
0$. For this purpose, let us write $M^{\epsilon,\gamma}_{T}=e^{\mathcal{E}%
^{\epsilon,\gamma}_{T}}$, where
\[
\mathcal{E}^{\epsilon,\gamma}_{T}=\frac{\delta}{\epsilon}\frac{1}{\sqrt{2}}
\int_{0}^{T}\phi(\delta^{2}s/\epsilon,x,\hat{Y}^{y_{0},\gamma}%
_{s},\gamma)dW_{s}-\frac{1}{2}\left(  \frac{\delta}{\epsilon}\right)
^{2}\frac{1}{2}\int_{0}^{T}\left\Vert  \phi(\delta^{2}s/\epsilon,x,\hat{Y}%
^{y_{0},\gamma}_{s},\gamma)\right\Vert  ^{2}ds
\]
Notice that
\[
\mathcal{E}^{\epsilon,\gamma}_{T}=N^{\epsilon,\gamma}_{T}-\frac{1}{2}\left<
N^{\epsilon,\gamma}\right>  _{T}
\]
where
\begin{align}
N^{\epsilon,\gamma}_{T}=\frac{\delta}{\epsilon}\frac{1}{\sqrt{2}} \int_{0}%
^{T}\phi(\delta^{2}s/\epsilon,x,\hat{Y}^{y_{0},\gamma}_{s}%
,\gamma)dW_{s}\nonumber
\end{align}

Since, $\phi$ is by assumption bounded, we obtain that $N^{\epsilon,\gamma
}_{T}$ is a continuous martingale and $\left<  N^{\epsilon,\gamma}\right>
_{T}$ is its quadratic variation. Boundedness of $\phi$ and the assumption
$\delta/\epsilon\downarrow0$ as $\epsilon\downarrow0$, implies that
\begin{align}
\lim_{\epsilon\downarrow0}\sup_{\gamma\in\Gamma}\mathbb{E}\left<  N^{\epsilon,\gamma
}\right>  _{T}  &  =\lim_{\epsilon\downarrow0}\sup_{\gamma\in\Gamma}\frac{1}{2}\left(  \frac{\delta
}{\epsilon}\right)  ^{2}\mathbb{E}\int_{0}^{T}\left\Vert  \phi(\delta
^{2}s/\epsilon,x,\hat{Y}^{y_{0},\gamma}_{s},\gamma)\right\Vert
^{2}ds\nonumber\\
&  =0. \label{Eq:QuadraticVariationOfExponent}%
\end{align}

Hence, uniformly in $\gamma\in\Gamma$, $\left<  N^{\epsilon,\gamma}\right>  _{T}$
converges to $0$  in probability and by Problem 1.9.2 in
\cite{LipsterShirayev1989}, the same convergence holds for the martingale
$N^{\epsilon,\gamma}_{T}$ as well. Thus, we have obtained that uniformly in
$\gamma\in\Gamma$
\begin{equation}
M^{\epsilon,\gamma}_{T}=e^{\mathcal{E}^{\epsilon,\gamma}_{t}}\text{ converges
to }1\text{ in probability, as }\epsilon
\downarrow0. \label{Eq:ExponentialmartingaleConvergence}%
\end{equation}

Moreover, (\ref{Eq:ExponentialmartingaleConvergence}) together with
Scheff\'{e}'s theorem (Theorem 16.12 in \cite{Billingsley}) imply that
\begin{equation}
\sup_{\gamma\in\Gamma}\mathbb{E}\left|  M^{\epsilon,\gamma}_{T}-1\right|  \rightarrow
0\text{, as }\epsilon\downarrow0.
\label{Eq:ExponentialmartingaleConvergenceL1}%
\end{equation}

In fact, boundedness of $\phi$ implies that for every $\epsilon\in(0,1)$ and
$\gamma\in\Gamma$, $M^{\epsilon,\gamma}_{T}$ is a square integrable
martingale. The latter statement and convergence in probability
(\ref{Eq:ExponentialmartingaleConvergence}), imply that
\begin{equation}
\sup_{\gamma\in\Gamma}\mathbb{E}\left|  M^{\epsilon,\gamma}_{T}-1\right|  ^{2}%
\rightarrow0\text{, as }\epsilon\downarrow0.
\label{Eq:ExponentialmartingaleConvergenceL2}%
\end{equation}

Now that (\ref{Eq:ExponentialmartingaleConvergenceL2}) has been established,
we continue with the proof of the lemma. Choose $h(\epsilon)$, such that
(\ref{Eq:ErgodicTheorem1}) holds, we obtain
\begin{align}
&  \mathbb{E}\left|  \frac{1}{h(\epsilon)}\int_{t}^{t+h(\epsilon
)}\Psi(s,x,Y_{s}^{\epsilon,x,y_{0},\gamma},\gamma)ds-\bar{\Psi
}(t,x)\right| \nonumber\\
&  \qquad= \mathbb{E}\left|  \left(  \frac{\delta^{2}/\epsilon
}{h(\epsilon)}\int_{\frac{t}{\delta^{2}/\epsilon}}^{\frac{t+h(\epsilon
)}{\delta^{2}/\epsilon}}\Psi(s,x,\hat{Y}^{y_{0},\gamma}_{s}%
,\gamma)ds\right)  M^{\epsilon,\gamma}_{T}-\bar{\Psi}(t,x)\right| \nonumber\\
&  \qquad\leq\mathbb{E}\left|  \frac{\delta^{2}/\epsilon}%
{h(\epsilon)}\int_{\frac{t}{\delta^{2}/\epsilon}}^{\frac{t+h(\epsilon)}%
{\delta^{2}/\epsilon}}\Psi(s,x,\hat{Y}^{y_{0},\gamma}_{s}%
,\gamma)ds-\bar{\Psi}(t,x)\right| \nonumber\\
&  \qquad\quad+\mathbb{E}\left|  \left(  \frac{\delta^{2}/\epsilon
}{h(\epsilon)}\int_{\frac{t}{\delta^{2}/\epsilon}}^{\frac{t+h(\epsilon
)}{\delta^{2}/\epsilon}}\Psi(s,x,\hat{Y}^{y_{0},\gamma}_{s}%
,\gamma)ds\right)  (M^{\epsilon,\gamma}_{T}-1)\right|\nonumber
\end{align}

Clearly, the first term converges to zero by (\ref{Eq:ErgodicTheorem1}). The
second term also converges to zero by H\"{o}lder's inequality, Lemma \ref{L:ErgodicTheorem1u}  applied to $\Psi^{2}$ and (\ref{Eq:ExponentialmartingaleConvergenceL2}).

The claim that the convergence is locally uniformly with respect to the parameter
$x\in\mathbb{R}^{m}$ follows by the fact that this is true for
(\ref{Eq:ErgodicTheorem1}).
\end{proof}

\subsection{Ergodic theorems with perturbation by small drift-Controlled case}

\begin{lemma}
\label{L:ErgodicTheorem3} Fix $T<\infty$ and consider $\mathcal{A}$ to
be the set of progressively measurable controls such that
\begin{equation}
\int_{0}^{T}\left\Vert u^{\epsilon}(s)\right\Vert^{2}ds< N , 
\label{Eq:UniformlySquareIntegrableControlsAdditional}%
\end{equation}
where the constant $N$ does not depend on $\epsilon,\delta, T$ or $\gamma$ and additionally such that for $\delta/\epsilon\ll 1$
\begin{equation}
\frac{1}{\epsilon}\mathbb{E}\int_{0}^{\frac{\delta^{2}T}{\epsilon}}\left\Vert u^{\epsilon}(s)\right\Vert^{2}ds\leq C \delta/\epsilon,
\label{Eq:UniformlySquareIntegrableControlsAdditionalb}%
\end{equation}
where the constant $C$ depends on $T$, but not on $\epsilon,\delta$ or $\gamma$. Consider
the process $\bar{Y}_{t}^{\epsilon,x,y_{0},\gamma}$ satisfying the SDE
\begin{align}
\bar{Y}_{t}^{\epsilon,x,y_{0},\gamma}  &  =y_{0}+\frac{\epsilon}{\delta^{2}%
}\int_{0}^{t}f(\bar{Y}_{s}^{\epsilon,x,y_{0},\gamma},\gamma)ds+\frac{1}%
{\delta}\int_{0}^{t}\left[  g(s,x,\bar{Y}_{s}^{\epsilon,x,y_{0},\gamma}%
,\gamma)+\kappa(\bar{Y}_{s}^{\epsilon,x,y_{0},\gamma},\gamma)
u^{\epsilon}(s)\right]  ds\nonumber\\
&  \qquad+\frac{\sqrt{\epsilon}}{\delta}\int_{0}^{t}\kappa(\bar{Y}%
_{s}^{\epsilon,x,y_{0},\gamma},\gamma)dW_{s}\nonumber
\end{align}

Let us consider a function $\tilde{\Psi}:[0,T]\times\mathbb{R}^{m}\times
\Gamma$ such that $\tilde{\Psi}(t,x,\cdot)\in L^{2}(\Gamma)\cap L^{1}(\pi)$
 and define $\Psi(t,x,y,\gamma)=\tilde{\Psi
}(t,x,\tau_{y}\gamma)$. We assume that the function $\Psi:[0,T]\times
\mathbb{R}^{m}\times\mathbb{R}^{d-m}\times\Gamma\mapsto\mathbb{R}$ is
measurable, piecewise constant in $t$ and uniformly continuous in $x$ with
respect to $(t,y)$.

Denote $\bar{\Psi}(t,x)\doteq\int_{\Gamma}\tilde{\Psi}(t,x,\gamma)\pi
(d\gamma)$ for all $(t,x)\in[0,T]\times\mathbb{R}^{m}$. Fix  $\eta>0$.  Then there exists
a set $N_{\eta}$ such that $\pi(N_{\eta})\geq 1-\eta$ and $h(\epsilon)\in \mathcal{H}_{1}^{N_{\eta}} $ such that
\begin{equation*}
\lim_{\epsilon\downarrow0}\sup_{\gamma\in N_{\eta}}\sup_{0\leq t\leq T}\mathbb{E}\left\vert
\frac{1}{h(\epsilon)}\int_{t}^{t+h(\epsilon)}\Psi(s,x,\bar{Y}_{s}%
^{\epsilon,x,y_{0},\gamma},\gamma)ds-\bar{\Psi}(t,x)\right\vert =0
\end{equation*}
locally uniformly with respect to the parameter $x\in\mathbb{R}^{m}$.
\end{lemma}

\begin{proof}
[Proof of Lemma \ref{L:ErgodicTheorem3}]Let us set $\hat{\bar{Y}}_{t}%
^{\epsilon,x,y_{0},\gamma}=\bar{Y}_{\delta^{2}t/\epsilon}^{\epsilon,x,y_{0}%
,\gamma}$. Notice that $\hat{\bar{Y}}_{t}^{\epsilon,x,y_{0},\gamma}$
satisfies
\begin{align}
\hat{\bar{Y}}_{t}^{\epsilon,x,y_{0},\gamma}  &  =y_{0}+\int_{0}^{t}f(\hat
{\bar{Y}}_{s}^{\epsilon,x,y_{0},\gamma},\gamma)ds+\frac{\delta}{\epsilon}%
\int_{0}^{t}\left[  g\left(  \frac{\delta^{2}}{\epsilon}s,x,\hat{\bar{Y}}%
_{s}^{\epsilon,x,y_{0},\gamma},\gamma\right)  +\kappa(\hat{\bar{Y}%
}_{s}^{\epsilon,x,y_{0},\gamma},\gamma) u^{\epsilon}\left(\delta^{2}s/\epsilon
\right)\right]  ds\nonumber\\
&  \qquad+\int_{0}^{t}\kappa(\hat{\bar{Y}}_{s}^{x,y_{0},\gamma
,\epsilon},\gamma)dW_{s}\nonumber
\end{align}

Essentially, based on the condition of the allowable controls
(\ref{Eq:UniformlySquareIntegrableControlsAdditional}), the arguments of the
uncontrolled case, Lemma \ref{L:ErgodicTheorem2}, go through verbatim. The
only place that needs some discussion is in regards to the proof of the
statement corresponding to (\ref{Eq:QuadraticVariationOfExponent}). Let us
show now how this term can be treated. In the controlled case we have that
\[
\phi(s,x,y,\gamma)=-\kappa^{-1}(y,\gamma)g(s,x,y,\gamma)-u^{\epsilon}(s)
\]
and we want to prove that for every $\gamma\in \Gamma$
\begin{align}
\lim_{\epsilon\downarrow0}\mathbb{E}\left<  N^{\epsilon,\gamma
}\right>  _{T}  &  =\lim_{\epsilon\downarrow0}\frac{1}{2}\left(  \frac{\delta
}{\epsilon}\right)  ^{2}\mathbb{E}\int_{0}^{T}\left\Vert  \phi(\delta
^{2}s/\epsilon,x,\hat{Y}^{y_{0},\gamma}_{s},\gamma)\right\Vert  ^{2}ds=0.
\label{Eq:QuadraticVariationOfExponentControlled}%
\end{align}

It is clear that
\begin{align}
\mathbb{E}\left<  N^{\epsilon,\gamma}\right>  _{T}  &  =\frac{1}%
{2}\left(  \frac{\delta}{\epsilon}\right)  ^{2}\mathbb{E}\int_{0}%
^{T}\left\Vert  \phi(\delta^{2}s/\epsilon,x,\hat{Y}^{y_{0}}_{s}%
,\gamma)\right\Vert  ^{2}ds\nonumber\\
&  \leq \left(  \frac{\delta}{\epsilon}\right)  ^{2}\mathbb{E}\int
_{0}^{T}\left\Vert  \kappa^{-1}(\hat{Y}^{y_{0}}_{s},\gamma)g(\delta
^{2}s/\epsilon,x,\hat{Y}^{y_{0}}_{s},\gamma)\right\Vert  ^{2}ds+ \left(
\frac{\delta}{\epsilon}\right)  ^{2}\mathbb{E}\int_{0}^{T}\left\Vert
u^{\epsilon}\left(\delta^{2}s/\epsilon\right)\right\Vert  ^{2}ds\nonumber
\end{align}
The first term of the right hand side of the last display goes to zero by the
boundedness of $\left\Vert  \kappa^{-1}g\right\Vert  ^{2}$ (as in Lemma
\ref{L:ErgodicTheorem2}). So we only need to consider the second term. Here we
use Condition \ref{Eq:UniformlySquareIntegrableControlsAdditionalb}. In
particular, we notice that Condition
\ref{Eq:UniformlySquareIntegrableControlsAdditionalb} gives
\begin{align}
\lim_{\epsilon\downarrow 0}\left(  \frac{\delta}{\epsilon}\right)  ^{2}\mathbb{E}\int_{0}^{T}\left\Vert  u^{\epsilon
}\left(\delta^{2}s/\epsilon\right)\right\Vert  ^{2}ds  &  =\lim_{\epsilon\downarrow 0}\frac{1}{\epsilon}\mathbb{E}\int
_{0}^{\delta^{2}T/\epsilon}\left\Vert  u^{\epsilon}(s)\right\Vert  ^{2}ds=0\nonumber
\end{align}
uniformly in $\gamma\in\Gamma$. Thus we have completed the
proof of (\ref{Eq:QuadraticVariationOfExponentControlled}). This concludes the
proof of the lemma.
\end{proof}

\subsection{Ergodic theorem with explicit dependence on the slow process. }

In this subsection we consider the pair $(\bar{X}_{s}^{\epsilon,\gamma},
\bar{Y}_{s}^{\epsilon,\gamma})$ satisfying (\ref{Eq:Main2}) and the purpose is to prove Lemma \ref{L:ErgodicTheorem4}.

\begin{lemma}
\label{L:ErgodicTheorem4} Consider the set-up, assumptions and notations of Lemma \ref{L:ErgodicTheorem3}. Fix  $\eta>0$.  Then there exists
a set $N_{\eta}$ such that $\pi(N_{\eta})\geq 1-\eta$ and $h(\epsilon)\in \mathcal{H}_{1}^{N_{\eta}} $ such that
\begin{equation*}
\lim_{\epsilon\downarrow0}\sup_{\gamma\in N_{\eta}}\sup_{0\leq t\leq T}\mathbb{E}\left\vert
\frac{1}{h(\epsilon)}\int_{t}^{t+h(\epsilon)}\Psi\left(s,\bar{X}_{s}^{\epsilon,\gamma},
\bar{Y}_{s}^{\epsilon,\gamma},\gamma\right)ds-\bar{\Psi}(t,\bar{X}_{t}^{\epsilon,\gamma})\right\vert =0.
\end{equation*}
\end{lemma}

\begin{proof}
[Sketch of proof of Lemma \ref{L:ErgodicTheorem4}]
  Due to Lemma \ref{L:ErgodicTheorem3}, the statement follows by using the standard argument of freezing the slow component, see for example Chapter 7.9 of \cite{FW1} or \cite{PardouxVeretennikov1}. Details are omitted.
\end{proof}

\section{Acknowledgements}
The author would like to thank Paul Dupuis for discussions on aspects of this work.  The author was partially supported by the National Science Foundation (DMS 1312124).

\end{document}